\documentclass{svmult}   
\usepackage{amssymb,amsmath,latexsym}
\usepackage[all]{xy}
\usepackage{graphicx}
\usepackage{newtxtext}
\usepackage{newtxmath}
\usepackage{makeidx}
\usepackage{multicol}
\usepackage{footmisc}
\spdefaulttheorem{examp}{Example}{\bf}{\rm}
\makeindex
\begin{document}
\title*{Wavelet sets for crystallographic groups}
\author{Kathy~D.
~Merrill}
\institute{Kathy D. Merrill \at Department of Mathematics, Colorado College, Colorado Springs, Colorado, 80903, USA,
\email{kmerrill@coloradocollege.edu}}
\maketitle
\abstract
{Single wavelet sets, and thus single wavelets, are shown to exist for the actions of all  crystallographic groups on $\mathbb R^2$ under all integer dilations.  Examples of such sets satisfying the additional requirement that they are finite unions of convex sets are presented for each of the groups under dilation by two. }   
 \keywords{Crystallographic group, Wavelet set;  MSC: 42C15; 42C40}
 %%%%%%%%%%%%%%%%%%%%%%%%%%
\section{Introduction}

Classical wavelets are finite sets of functions in $L^2(\mathbb R^n)$ whose integer translations and expansive matrix dilations give an orthonormal basis for the space.  A generalization developed in \cite{bcmo} replaces $L^2(\mathbb R^n)$ with an abstract Hilbert space, and uses a group action and powers of a compatible operator in place of the integer translations and matrix dilations.  In this paper, we consider such wavelets using the action of a 2-dimensional crystallographic group\index{crystallographic group} $\mathcal G$ on $L^2(\mathbb R^2)$ together with an integer dilation $\delta$.  Examples of crystallographic wavelets\index{crystallographic wavelet} were first given in \cite{gllww2}, \cite{gllww1}, and \cite{gllww3} under the rubric of composite dilations\index{composite dilation}.  The theory of these  wavelets was later worked out in general for all 2-dimensional crystallographic groups in \cite{mac} and \cite{mt}.  Recent related work by Barbieri, Cabrelli, Hernandez, and Molter (\cite{bchm1},\cite{bchm2}) has successfully used subspaces invariant under crystallographic groups to approximate large data sets.  
\par Many of the examples of wavelets for crystallographic groups given in the literature are Haar-type wavelets\index{Haar-type wavelet}, that is, wavelets that are normalized characteristic functions of sets.  (See e.g. \cite{mac}, \cite{mt}, \cite{gm}, \cite{bs}, \cite{bk}, \cite{krww}.) In this paper, we focus instead on wavelet set \index{wavelet set} wavelets, whose Fourier transforms are characteristic functions. Specifically, write $\mathcal F$ for the 2-dimensional Fourier transform,
\begin{equation*}
\hat f(\xi)=(\mathcal Ff)(\xi)=\int_{\mathbb R^2}f(z)e^{-2\pi i\langle z,\xi\rangle}dz,  
\end{equation*}
and $\mathcal F^{-1} f=\check{f}$ for the inverse Fourier transform. 
\begin{definition}
A finite collection of sets  $\{W_1,\cdots, W_L\}$ is called an $L$-wavelet set if $\{\mathcal F^{-1}({\bf 1}_{W_1}),\cdots,\mathcal F^{-1}({\bf 1}_{W_L})\}$ is a wavelet, so that 
$\{\delta^j(\gamma(\mathcal F^{-1}({\bf 1}_{W_l})))\}$ forms an orthonormal basis for $L^2(\mathbb R^2)$, with $-\infty<j<\infty,\;\gamma\in\mathcal G,$ and $1\leq l\leq L$.
\end{definition}
For us, $\mathcal G$ will be a two dimensional crystallographic group, and $\delta$ an integer scalar dilation.  A set $W$ is called a single wavelet set, or simply a {\it wavelet set}\index{wavelet set} if $L=1$; it is called a {\it multiwavelet set}\index{multiwavelet set} if $L>1$.  Examples of (multi)wavelet sets for some of the semidirect product\index{semidirect product} crystallographic groups were given in \cite{gllww2} and \cite{gllww1}.  An example of a multiwavelet set for a crystallographic group that is not a semidirect product appears in \cite{mer2}.  
\par All of these known examples of wavelet sets for the two dimensional crystallographic groups with integer dilations  have been multiwavelets.  In fact, all of the integer dilation crystallographic wavelets\index{crystallographic wavelet} of any type that have appeared in the literature are multiwavelets.  This is similar to the early history of classical wavelets in $L^2(\mathbb R^2)$, before Dai, Larson and Speegle \cite{dls} showed in the late 1990's that classical single wavelet sets, and thus single wavelets, exist for all dilations.  In this paper we prove an analogous result for the crystallographic groups.  That is, we show that single wavelet sets, and thus single wavelets, exist for all the two dimensional crystallographic groups for all integer dilations.  Our result does not require that the integer dilation $\delta$ be compatible\index{compatible dilation} with the crystallographic group $\mathcal G$ in the sense developed in \cite{bcmo},  \cite{gllww2}, and \cite{mt}.    For dilations that do satisfy the compatibility condition, it is shown in \cite{bmpt} that single wavelet sets can be used to decompose the wavelet representations of the crystallographic groups into direct integrals\index{direct integral} of irreducibles. (For more about the compatibility condition, see Section \ref{prelim}.) 
\par The general process for building single wavelet sets for the crystallographic groups depends on the same iterative methods used for classical single wavelet sets.  Thus, as in the classical case, the crystallographic examples built using the existence theorem have a complicated geometric structure requiring an infinite number of convex pieces.  However, just as in the classical case \cite{mer1}, simple single wavelet sets\index{simple wavelet set}, i.e. sets that consist of a finite number of convex pieces, can be found in many cases.   In this paper, we find simple single wavelet sets for all crystallographic groups under dilation by two.  Wavelets based on these sets are fairly simple to encode for image compression.  Keith Taylor \cite{kt} reports supervising an honors thesis that performed such computations by breaking these wavelet sets into triangles.  Of course the slow decay of the associated wavelets prevents rapid convergence, but such wavelets are of theoretical interest, and may also be easier to smooth because of their simple nature.  (See e.g. \cite{mer3}.)    
\par  The paper is organized as follows.  In Section 2, we review needed facts about the crystallographic groups.  Section 3 develops necessary and sufficient conditions for a set to be a crystallographic wavelet set, and contains the proof that such sets always exist for integer dilations.  Section 4 has examples of simple single wavelet sets under dilation by two for all 17 crystallographic groups.     
%%%%%%%%%%%%%%%%%%%%%%%%%%%%
\section{Preliminaries}
\label{prelim}
A  two dimensional crystallographic group\index{crystallographic group} $\mathcal G$, also called a wallpaper group\index{wallpaper group}, is a discrete subgroup of $\mathbb R^2\rtimes \mathcal O_2$, where $\mathcal O_2$ is the $2-$dimensional orthogonal group.  An element $[x,L]\in\mathcal G$, with $x\in\mathbb R^2$ and $L\in \mathcal O_2$, acts on $\mathbb R^2$ by $[x,L]\cdot z=L(x+z)$, and on $L^2(\mathbb R^2)$ by $[x,L]\cdot f(z)=f(L^{-1}z-x)$.  Multiplication in the group is given by
 $$[x,L]\cdot[y,M]=[M^{-1}x+y,LM].$$  
Such a group represents a possible symmetry pattern in the plane that is repeated in two independent directions.  Accordingly, each crystallographic group $\mathcal G$  has a normal subgroup $\mathcal N$ consisting of a lattice\index{lattice} of translations isomorphic to $\mathbb Z^2$.  The quotient $\mathcal G/{\mathcal N}$ is isomorphic to $\mathcal B=\{S\in\mathcal O_2: [x,S]\in\mathcal G \mbox{ for some } x\in\mathbb R^2\}$, which is a finite group called the point group\index{point group}.

\par There are 17 two dimensional crystallographic groups \cite{fed}.  This can be shown by first proving that the only possible elements of the point groups are rotations by $\frac{2\pi}n$, for $n=1,2,3,4,6,$ and reflections. Thirteen of the groups are semidirect products\index{semidirect product} of their point group and lattice.  Such crystallographic groups are called symmorphic\index{symmorphic}.  The four non-symmorphic groups have essential glide reflections\index{glide reflection}, which are simultaneous reflections and non-lattice translations.  These glides prevent the separation of the action of the lattice and that of the point group.  The traditional names of the wallpaper groups, which were first introduced by crystallographers, are sequences of four symbols specifying cell type, highest order rotation, and essential reflections and glides.  In this paper, we will use a common shortened form of these names. (See e.g. \cite{sch} for a description of all 17 wallpaper groups and an explanation of their names.) 

\par Table \ref{point} gives explicit descriptions of all the wallpaper groups and their point groups in terms of one possible set of generators for each.  These descriptions depend on a specific choice of lattice.  We choose $\mathbb Z^2$ whenever it is possible to do so, thus for all groups not containing rotation by 3 or 6.  In the latter case, we choose the lattice $\mathcal L$ generated by $(1,0)$ and $(\frac12,\frac{\sqrt3}2)$.  To make the lattice apparent, we include two independent translations in the list of generators in the table, even if this causes redundancy among the generators.  We write $\rho_n$ for rotation ccw by $\frac{2\pi}n$, $\tau_v$ for translation by $v$, $\sigma_v$ for reflection in the line determined by $v$, and $I$ for the $2\times 2$ identity matrix.  The
non-symmorphic groups in Table \ref{point} are pg, pmg, pgg, and p4g; these groups include glides as generators.  The information presented in the table, as well as more detail about the groups, can be found, for example, in \cite{far} or \cite{mor}.  

\begin{table}
\begin{tabular}{|c|c|c|} \hline
Name& Group Description&Point Group\index{point group}\\ \hline\hline
p1 &$\langle\tau_{(1,0)},\;\tau_{(0,1)}\rangle$& $\langle I\rangle$   \\ \hline
p2 &$\langle\tau_{(1,0)},\;\tau_{(0,1)},\;\rho_2\rangle$& $\langle\rho_2\rangle$\\ \hline
pm &$\langle\tau_{(1,0)}\;,\tau_{(0,1)},\;\sigma_{(0,1)}\rangle$& $\langle\sigma_{(0,1)}\rangle$ \\ \hline
pg &$\langle\tau_{(1,0)},\;\tau_{(0,1)},\;[(0,\frac12),\sigma_{(0,1)}]\rangle$& $\langle\sigma_{(0,1)}\rangle$ \\ \hline
pmm &$\langle\tau_{(1,0)},\;\tau_{(0,1)},\;\rho_2,\;\sigma_{(0,1)}\rangle$& $\langle\rho_2,\;\sigma_{(0,1)}\rangle$\\ \hline
pmg &$\langle\tau_{(1,0)},\;\tau_{(0,1)},\;\rho_2,\;[(0,\frac12),\sigma_{(0,1)}]\rangle$& $\langle\rho_2,\;\sigma_{(0,1)}\rangle$\\ \hline
pgg &$\langle\tau_{(1,0)},\;\tau_{(0,1)},\;\rho_2,\;[(\frac12,\frac12),\sigma_{(0,1)}]\rangle$& $\langle\rho_2,\;\sigma_{(0,1)}\rangle$\\ \hline
p4 &$\langle\tau_{(1,0)},\;\tau_{(0,1)},\;\rho_4\rangle$& $\langle\rho_4\rangle$\\  \hline
p4m &$\langle\tau_{(1,0)},\;\tau_{(0,1)},\;\rho_4,\;\sigma_{(1,1)}\rangle$& $\langle\rho_4,\;\sigma_{(1,1)}\rangle$\\ \hline
p4g &$\langle\tau_{(1,0)},\;\tau_{(0,1)},\;\rho_4,\;[(\frac12,\frac12),\sigma_{(1,1)}]\rangle$& $\langle\rho_4,\;\sigma_{(1,1)}\rangle$\\ \hline
cm &$\langle\tau_{(1,0)},\;\tau_{(0,1)},\;\sigma_{(1,1)}\rangle$& $\langle\sigma_{(1,1)}\rangle$\\ \hline
cmm &$\langle\tau_{(1,0)},\;\tau_{(0,1)},\;\rho_2,\;\sigma_{(1,1)}\rangle$& $\langle\rho_2,\;\sigma_{(1,1)}\rangle$\\ \hline
p3 &$\langle\tau_{(1,0)},\;\tau_{(\frac12,\frac{\sqrt 3}2)},\;\rho_3\rangle$& $\langle\rho_3\rangle$\\ \hline
p31m &$\langle\tau_{(1,0)},\;\tau_{(\frac12,\frac{\sqrt 3}2)},\;\rho_3,\;\sigma_{(1,0)}\rangle$& $\langle\rho_3,\;\sigma_{(1,0)}\rangle$ \\ \hline
p3m1 &$\langle\tau_{(1,0)},\;\tau_{(\frac12,\frac{\sqrt 3}2)},\;\rho_3,\;\sigma_{(0,1)}\rangle$& $\langle\rho_3,\;\sigma_{(0,1)}\rangle$ \\ \hline
p6 &$\langle\tau_{(1,0)},\;\tau_{(\frac12,\frac{\sqrt 3}2)},\;\rho_6\rangle$& $\langle\rho_6\rangle$ \\ \hline
p6m &$\langle\tau_{(1,0)},\;\tau_{(\frac12,\frac{\sqrt 3}2)},\;\rho_6,\;\sigma_{(\frac{\sqrt3}2,\frac12)}\rangle$& $\langle\rho_6,\;\sigma_{(\frac{\sqrt3}2,\frac12)}\rangle$  \\ \hline

\end{tabular}
\caption{The two dimensional crystallographic groups\index{crystallographic group}}
\label{point}
\end{table}

Since we will be discussing wavelets for the two dimensional crystallographic groups, we will be interested in the interaction between the groups and a dilation.  We will restrict our attention to integer dilations.  For $d\in\mathbb Z$, $d\geq 2$, write $\delta_d$ for the dilation operator acting on $L^2(\mathbb R^2)$ by
$$\delta_d f(z)=d f(dz).$$
Other papers in the literature about crystallographic wavelets have raised the issue of compatibility between dilations and the group action.  For $\delta_d$ to be {\it compatible}\index{compatible dilation} with a crystallographic group $\mathcal G$ requires that $\delta_d^{-1}\mathcal G \delta_d\subset \mathcal G$.  This condition ensures that the group acts in a coherent way on the layers of the associated multiresolution structure.  (See e.g. \cite{bcmo},\cite{mer2}.)  It turns out that all integer dilations are compatible with symmorphic groups, but only odd dilations are compatible with the non-symmorphic groups \cite{mac}, \cite{mt}, \cite{bmpt}.   However, as this paper is concerned only with creating wavelet sets, the issue of compatibility does not arise.  Thus, we will consider all integer dilations.   
%%%%%%%%%%%%%%%%%
\section{Crystallographic wavelet sets}
\label{cws}
A well known result for conventional wavelet sets states that $W$ is a wavelet set for dilation by $d$ if and only if $W$ tiles the plane almost everywhere under both translation by the integers and dilation by $d$.  A corresponding result for crystallographic groups is given by the following theorem.   

\begin{theorem}
\label{wscond}
Let $\mathcal G$ be a crystallographic group with lattice $\mathcal N$ and point group $\mathcal B$\index{point group}.    Then $W$ is a wavelet set for $\mathcal G$\index{crystallographic wavelet set} and dilation by $d$ if and only if all of the following conditions hold
\begin{itemize}
\item[(i)] $W$ tiles\index{tiling} $\mathbb R^2$ a.e. under translation by $\mathcal N$; that is,
$$\sum_{k\in\mathcal N}{\bf1}_{W}(\xi+k)=1\quad a.e.\; \xi\in\mathbb R^2.$$
\item[(ii)] $\;S(W)\cap S'(W)$ has measure $0$ for $S\neq S'\in \mathcal B$
\item[(iii)] $\;\;\cup_{S\in \mathcal B}S(W)$ tiles $\mathbb R^2$ a.e. under dilation by $d$; that is,
$$\sum_{j\in\mathbb Z}{\bf1}_{\cup_{S\in \mathcal B}S(W)}(d^j\xi)=1\quad a.e.\; \xi\in\mathbb R^2.$$   
\end{itemize}
\end{theorem}
\begin{proof}
\par Writing $\widehat\psi={\bf1}_W$, we will show that $\{\widehat\delta^j\widehat\gamma\widehat\psi:j\in\mathbb Z, \gamma\in\mathcal G\}$ is an orthonormal basis for $L^2(\mathbb R^2)$ if and only if the three conditions hold.   
A straightforward calculation yields
\begin{equation}
\label{ftxS}
([x,S]^{\wedge}f)(\xi)=e^{-2\pi i\langle x,S^*(\xi)\rangle}f(S^*(\xi)).
\end{equation}   
Thus we see that $\{[\ell,I]^{\wedge}\widehat\psi:\ell\in\mathcal N\}=\{e^{2\pi i \langle \ell,\cdot\rangle}{\bf 1}_W:\ell\in\mathcal N\}$ gives an orthonormal basis for $L^2(W)$ if and only if (i) holds.  

We write $\mathcal G=\{[\ell+c_S,S]:\ell\in\mathcal N,S\in \mathcal B\}$, where $c_S\in\mathbb R^2\setminus\mathcal N$ is a constant for each $S$, and $c_S=0$ if $S=I$.  Applying Equation (1) for a fixed $S\in\mathcal B$, we have
 \begin{equation*}
 [\ell+c_S,S]^{\wedge}\widehat\psi(\xi)=e^{-2\pi i\langle S(c_S),\xi\rangle}e^{-2\pi i\langle \ell,S^*(\xi)\rangle}{\bf 1}_{W}(S^*(\xi)).
 \end{equation*}
This shows that condition (i) is equivalent to  $\{[\ell+c_S,S]^{\wedge}\widehat\psi:\ell\in\mathcal N\}$ forming an orthonormal basis for $L^2(S(W))$.
 
\par Using this, we see that conditions (i) and (ii) together imply that $\{\widehat\gamma\widehat\psi:\gamma\in\mathcal G\}$ forms an orthonormal basis for $L^2(\cup_{S\in \mathcal B}S(W))$.  Since
\begin{equation}
\label{ftdelta}
(\widehat\delta f)(\xi)=\delta^{-1}f(\xi).
\end{equation}   
we see that including condition (iii) implies 
that $\{\widehat\delta^j\widehat\gamma\widehat\psi:j\in\mathbb Z, \gamma\in\mathcal G\}$ is an orthonormal basis for $L^2(\mathbb R^2)$.

Conversely, if $W$ is a wavelet set for $\mathcal G$, the orthonormality of $\{\widehat\gamma\widehat\psi:\gamma\in\mathcal G\}$ implies condition (ii), and the orthonormality and completeness of  $\{\widehat\delta^j\widehat\gamma\widehat\psi:j\in\mathbb Z, \gamma\in\mathcal G\}$ implies both (iii) and the fact that $\{[\ell+c_S,S]^{\wedge}\widehat\psi:\ell\in\mathcal N\}$ forms an orthonormal basis for $L^2(S(W))$.   We have shown the latter to be equivalent to condition (i).
\end{proof}
\begin{remark} A similar proof shows that ${\bf1}_{W_1}, {\bf 1}_{W_2},\dots{\bf1}_{W_L}$ are Fourier transforms of an $L-$wavelet\index{multiwavelet set} for $\mathcal G$ and dilation by $d$ if and only if all of the following hold:
\begin{itemize}
\item[(i)] Each $W_l$ tiles $\mathbb R^2$ a.e. under translation by $\mathcal N$
\item[(ii)] $\;S(W_l)\cap S'(W_{l'})$ has measure 0 unless $S=S'$ and $l=l'$
\item[(iii)] $\;\;\cup_{l=1}^L\cup_{S\in \mathcal B}S(W)$ tiles $\mathbb R^2$ a.e. under dilation by $d$.  
\end{itemize}
\end{remark}
\begin{remark}
In the case of symmorphic\index{symmorphic} (semi-direct product) crystallographic groups, the theory of composite dilations applies.  For this case, Gu, Labate, Lim, Weiss, and Wilson \cite{gllww3} give necessary and sufficient conditions for Parseval frame\index{Parseval wavelet} multiwavelet sets\index{multiwavelet set} that are similar to conditions (i) and (iii) of Theorem \ref{wscond}.  Condition (ii) is not needed because their theorem describes Parseval wavelets rather than orthonormal wavelets.  The authors also give a strategy for building orthonormal multiwavelet sets.  They show this strategy cannot create single wavelet sets for reasons similar to those that prevent classical single MRA wavelets in $L^2(\mathbb R^2).$ 
\end{remark}

\begin{corollary}
\label{glide}
If two crystallographic groups $\mathcal G$ and $\mathcal G'$ differ in replacing reflections by glide reflections\index{glide reflection} that mirror over the same line, then $\mathcal G$ and $\mathcal G'$ have the same wavelet sets.  
\end{corollary}

Using Theorem \ref{wscond}, we can build single wavelet sets for crystallographic groups from conventional subspace wavelet sets\index{subspace wavelet set}.  Let $M\subset\mathbb R^2$ be a finite union of polar sectors of the form $\{r(\cos\phi, \sin\phi):\theta_1\leq\phi\leq\theta_2, r\geq0\}$, and let $\mathcal H$ be the Hilbert space defined by $\mathcal F(\mathcal H)=L^2(\mathbb R^2){\bf1}_{M}$.  A set $W\subset M$ is called a {\it subspace wavelet set} for $\mathcal H$ if ${\bf 1}_W$ is the Fourier transform of a wavelet for $\mathcal H$.   A more general form of such sets appears throughout the literature, and their properties are thoroughly  described in \cite{ddgh}.  For example, a subspace wavelet set $W$ of the form above has the property that its integer translates tile\index{tiling} $\mathbb R^2$, while its dilates by $d$ tile $M$.  When combined with information about the point group given in Table \ref{point}, such a set is perfectly situated to satisfy Theorem \ref{wscond}.  Thus, the result from \cite{ddgh} that subspace wavelet sets always exist can be used to prove the following theorem.  

\newsavebox{\rhomat}
\savebox{\rhomat}{ $\left(\begin{smallmatrix}
1&\frac12\\
0&\frac{\sqrt3}2
\end{smallmatrix}\right)$}
\newsavebox{\rhoamat}
\savebox{\rhoamat}{ $\left(\begin{smallmatrix}
1&-\frac12\\
0&\frac{\sqrt3}2
\end{smallmatrix}\right)$}    

\begin{theorem}
\label{existence}
Single wavelet sets exist for each of the 2-dimensional crystallographic groups\index{crystallographic wavelet set} for any integer dilation $d$.
\end{theorem}
\begin{proof} By the general theorem in \cite{dls}, as well as numerous constructions in, e.g. \cite{dls},\cite{bmm},\cite{bl}, we know that conventional single wavelet sets exist for all dilations.  These sets in turn are crystallographic wavelet sets for the group p1.  For the other crystallographic groups, we will use conventional subspace wavelet sets, which are shown to exist for all dilations and dilation invariant subspaces in \cite{ddgh}.
\par First, let $\mathcal G$ be a 2-dimensional crystallographic group with lattice $\mathbb Z^2$ and point group $\mathcal B$.  If $M$ is a polar sector that tiles the plane under the action of $\mathcal B$, and if $W$ is a subspace wavelet set for $M$, then by Theorem \ref{wscond}, $W$ is a crystallographic wavelet set for $\mathcal G$.  Specifically, for the groups p2, pm and pg, we can take $M$ to be the right half plane; for the group cm, we take $M$ to be one of the half planes determined by the line at angle $\frac{\pi}4$ above the horizontal axis; for pmm, pmg, pgg, and p4, $M$ is the first quadrant; for cmm, the quadrant between $\theta_1=-\frac{\pi}4$ and $\theta_2=\frac{\pi}4$; and for p4m and p4g, the first octant.  

The last 5 groups in Table \ref{point}, which involve 3-fold or 6-fold rotation, contain translations by the hexagonal lattice $\mathcal L$ determined by integer combinations of $(1,0)$ and $(\frac12,\frac{\sqrt 3}2)$, rather than by $\mathbb Z^2$.  Define $L=\usebox{\rhomat}$ and $L'=\usebox{\rhoamat}$.  Then if the set $W$ tiles under translation by the integers, both $LW$ and $L'W$ tile under translation by the hexagonal lattice\index{lattice}.  If $W$ tiles the polar sector $\theta\in[0,\theta_2]$ under dilation, $LW$ tiles $\theta\in[0,\frac{2\theta_2}3]$, and $L'W$ tiles $\theta\in[0,\frac{4\theta_2}3]$.  Thus, if $W$ is a subspace wavelet set for the first quadrant, $LW$ and $L'W$ are crystallographic wavelet sets for p6 and p3 respectively.  If $W$ is a subspace wavelet set for the first octant, $LW$ and $L'W$ are crystallographic wavelet sets for p6m and p31m respectively.  Neither $LW$ nor $L'W$ gives a wavelet set for p3m1 since in the first case, the union of the point group images do not tile $\mathbb R^2$ by dilation, and in the second, the point group images are not disjoint.  However, if we instead let $W$ be a subspace wavelet for the union of the first and fifth octants, then $LW$ does satisfy all three conditions of Theorem \ref{wscond} for p3m1.     
\end{proof}

\begin{examp}
We use the iterative method from \cite{bmm}, Theorem 3, to create a subspace wavelet set\index{subspace wavelet set} for the first quadrant under dilation by 3.  As in \cite{bmm}, we first build a scaling set $E$ such that $W=3E\setminus E$ is a subspace wavelet set for the first quadrant.  We start the iterative algorithm with $E_0=[0,\frac13)^2$.  Each successive step scales the previous piece by $\frac13$, and then translates it by $(\frac13,\frac13)$.  The resulting scaling set is $E=\cup_{n=1} ^{\infty}\left([0,\frac1{3^{n}})^2+(\sum_{j=1}^{n-1}\frac1{3^j},\sum_{j=1}^{n-1}\frac1{3^j})\right)$.  The conventional subspace wavelet set, $W=3E\setminus E$, is shown on the left in Figure \ref{quad1}. As mentioned in the proof of Theorem \ref{existence}, this is a crystallographic wavelet set\index{crystallographic wavelet set}  for pmm, pmg, pgg and p4.  We create wavelet sets for p3 and p6, by multiplying $W$ by the lattice matrices $L'=\usebox{\rhoamat}$ and $L=\usebox{\rhomat}$ respectively.  The results are shown in the center and  right of  Figure \ref{quad1}.        

\begin{figure}[h]
\centering $\begin{array}{ccc}
 \setlength{\unitlength}{80bp}
\begin{picture}(1,1)(.3,0)
\put(0,0){\includegraphics[width=\unitlength]{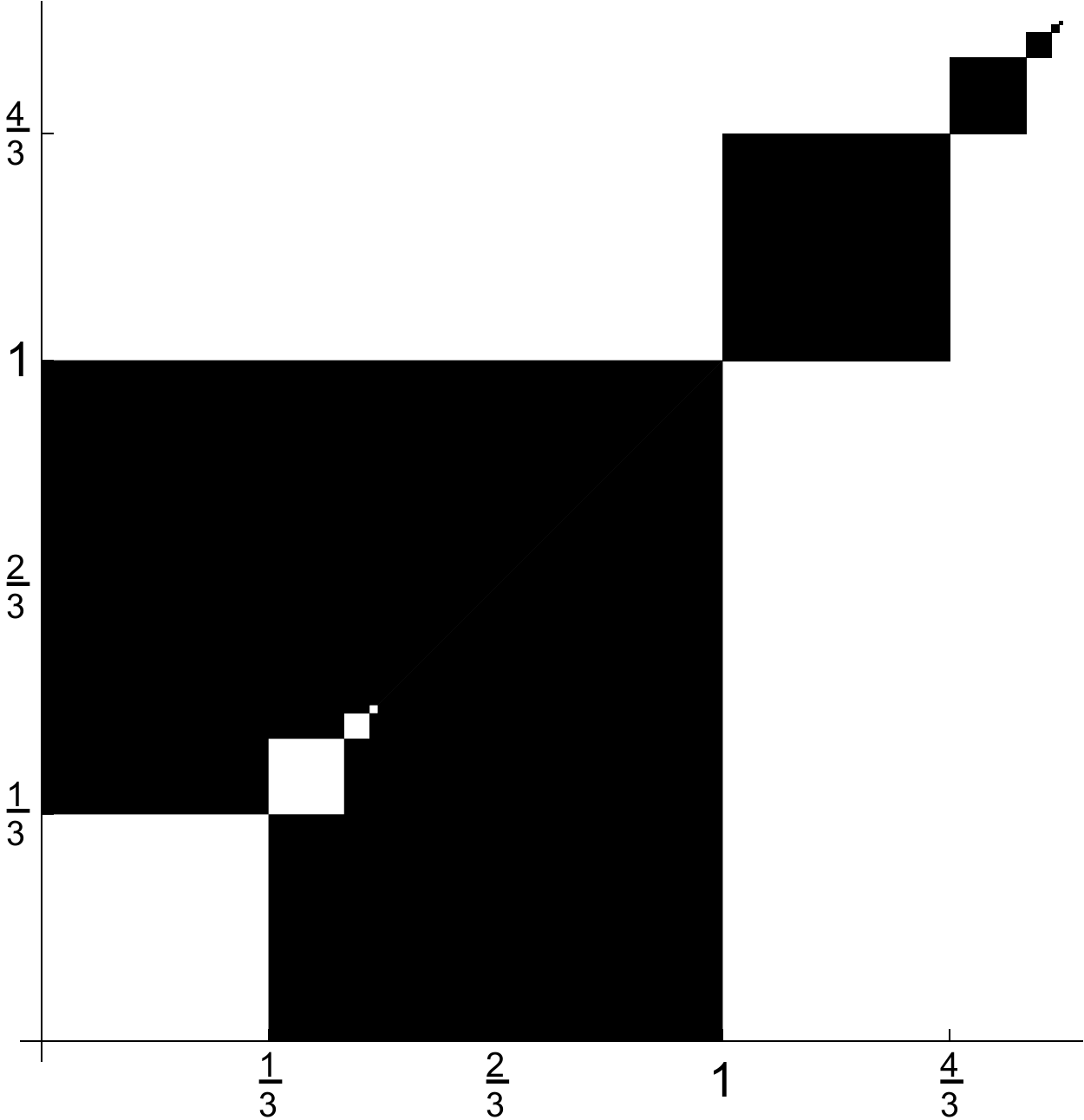}}
\end{picture}&
\setlength{\unitlength}{80bp}
\begin{picture}(1,1)(.1,0)
\put(0,0){\includegraphics[width=\unitlength]{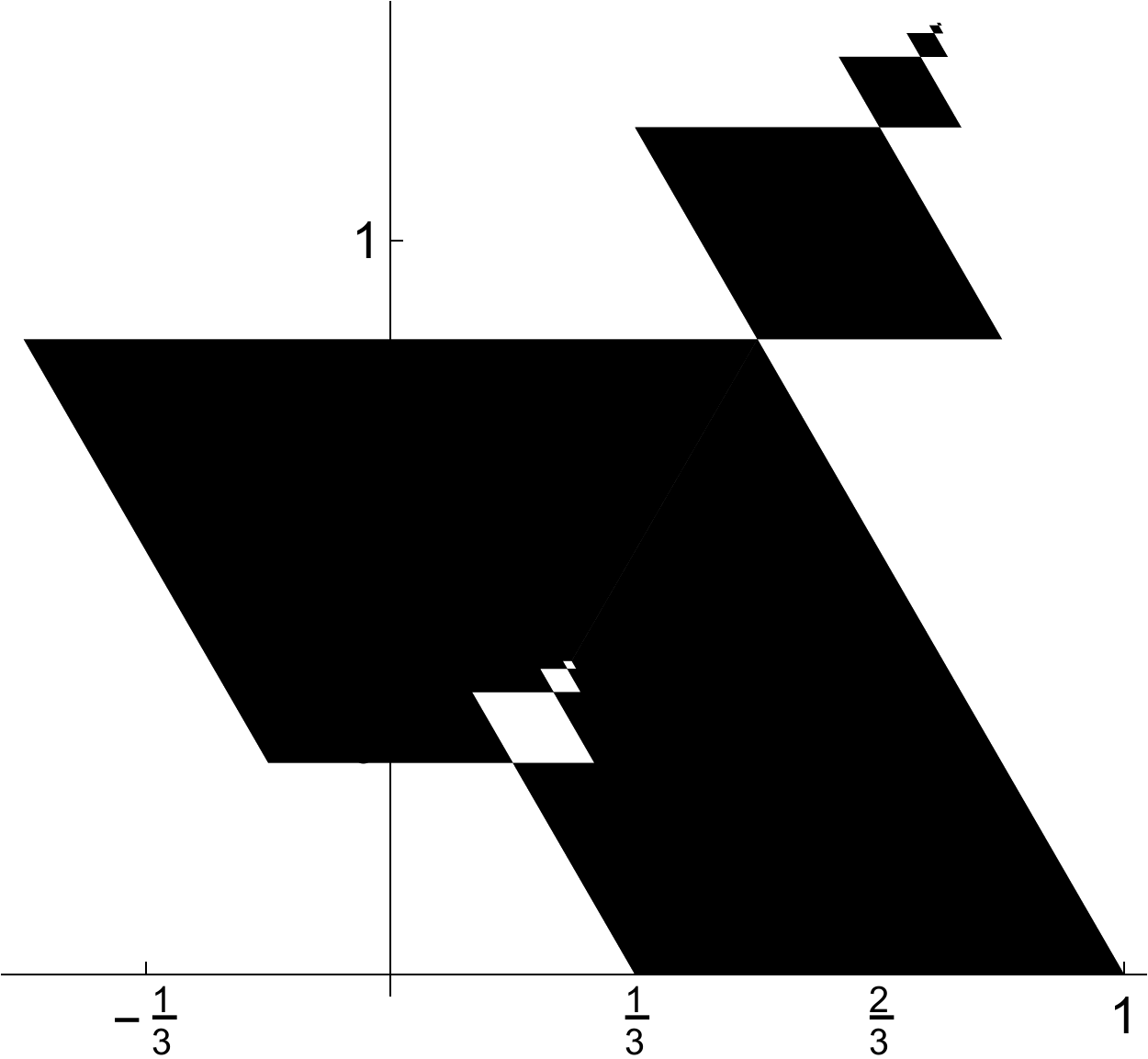}}
\end{picture}&
\setlength{\unitlength}{120bp}
\begin{picture}(1,1)(-.1,0)
\put(0,0){\includegraphics[width=\unitlength]{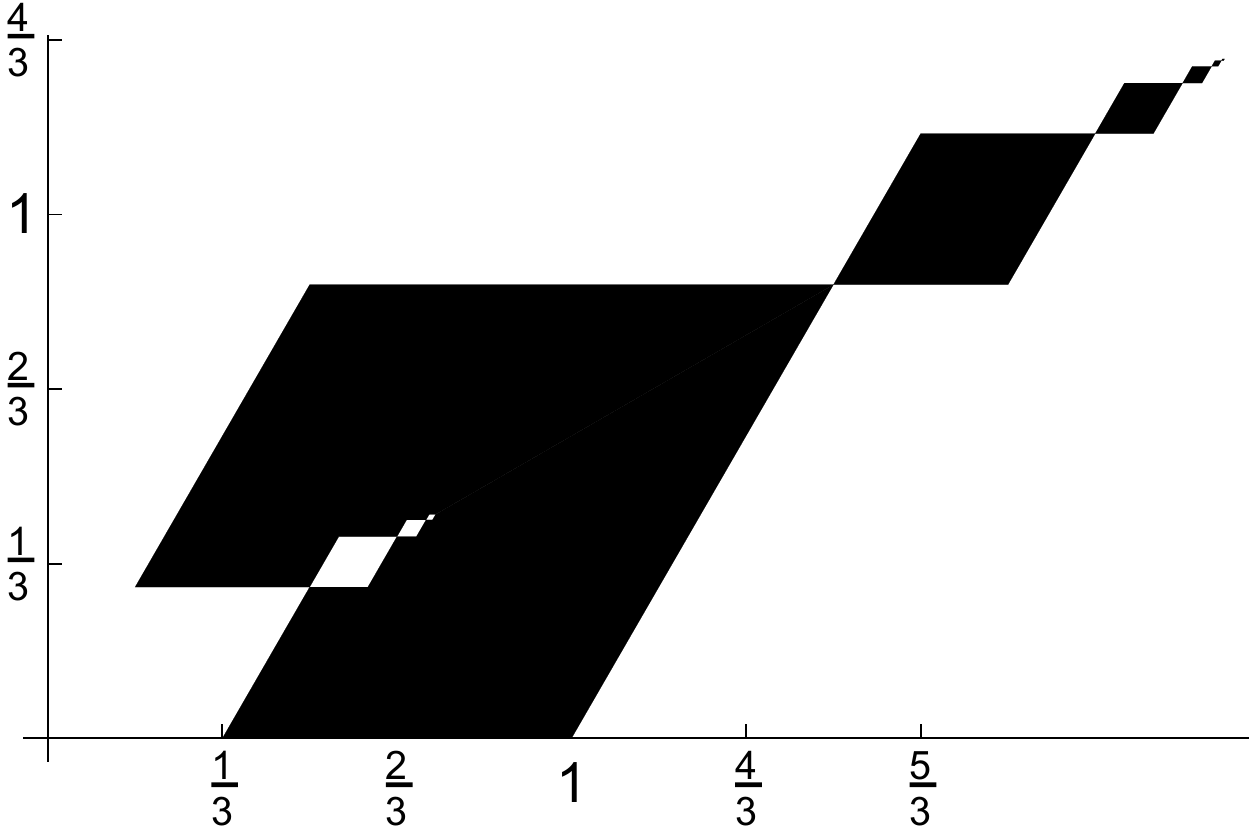}}
\end{picture}
\end{array}$
\caption{Dilation by 3 wavelet set for p4, pmm, pmg and pgg on the left,  for p3 in the center, and for p6 on the right.}
\label{quad1}
\end{figure}
\end{examp}

In the next section, we will find simple wavelet sets, that is, wavelet sets that are a finite union of convex sets, for all 17 crystallographic groups under dilation by 2.  Many of the sets are easily modified to other dilations.  In particular, the constructions below can be modified to give simple wavelet sets for dilation by 3, the smallest dilation that is compatible\index{compatible dilation} with all 17 crystallographic groups, for all but the groups p3 and p3m1.    
%%%%%%%%%%%%%%%%%%%%%%%%%%%%%%%%%%%%%%%
\section{Simple crystallographic wavelet sets}
\label{simple}
We first remark that crystallographic wavelet sets for the group p1 are the same as conventional wavelet sets.  Simple conventional wavelet sets\index{simple wavelet set}, and thus simple crystallographic wavelet sets for p1, are given in \cite{mer1} for any integer dilation $d$.  We will now use the ideas from the previous section to find simple wavelet sets under dilation by 2 for the other 16 groups.  Simple conventional subspace wavelet sets are sometimes difficult to find for the polar regions described in Theorem \ref{existence}.  Thus, in some cases, we will instead need to break up conventional simple multiwavelet sets\index{multiwavelet set}. 

\par For example, for crystallographic groups whose point groups have order 2, we find simple single wavelet sets by partitioning conventional 2-wavelet sets.  Such sets tile under dilation and give 2-fold tilings under translation\index{tiling}.  The sets in Figure \ref{2wave} were shown to be 2-wavelet sets for translation by $\mathbb Z^2$ and dilation by 2 in \cite{mer2}. 

\begin{figure}[h]
\centering $\begin{array}{cc}
 \setlength{\unitlength}{170bp}
\begin{picture}(1,1)(.1,-.2)
\put(0,0){\includegraphics[width=\unitlength]{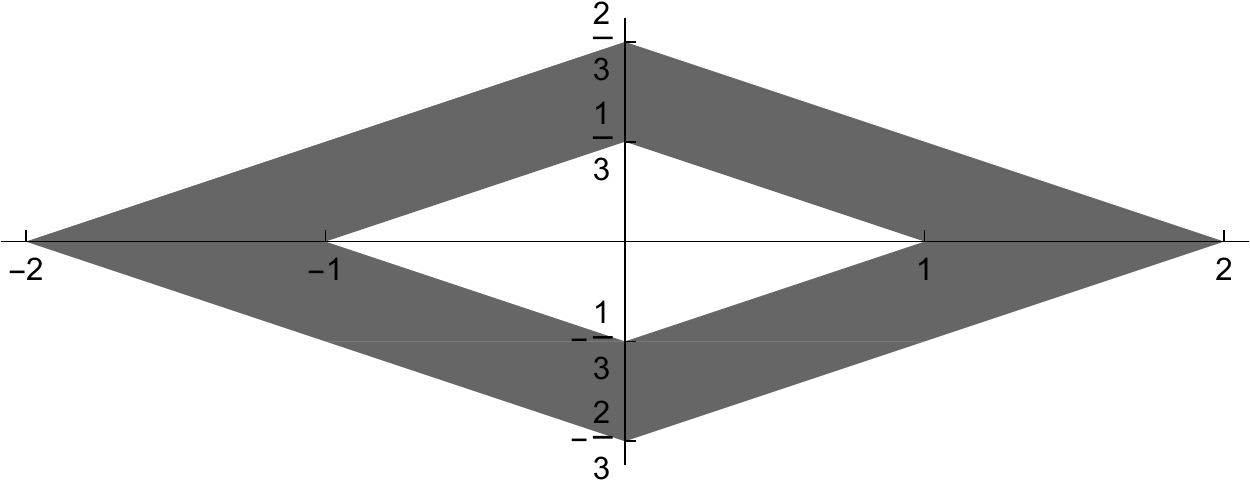}}
\end{picture}&
\setlength{\unitlength}{110bp}
\begin{picture}(1,1)(0,-.1)
\put(0,0){\includegraphics[width=\unitlength]{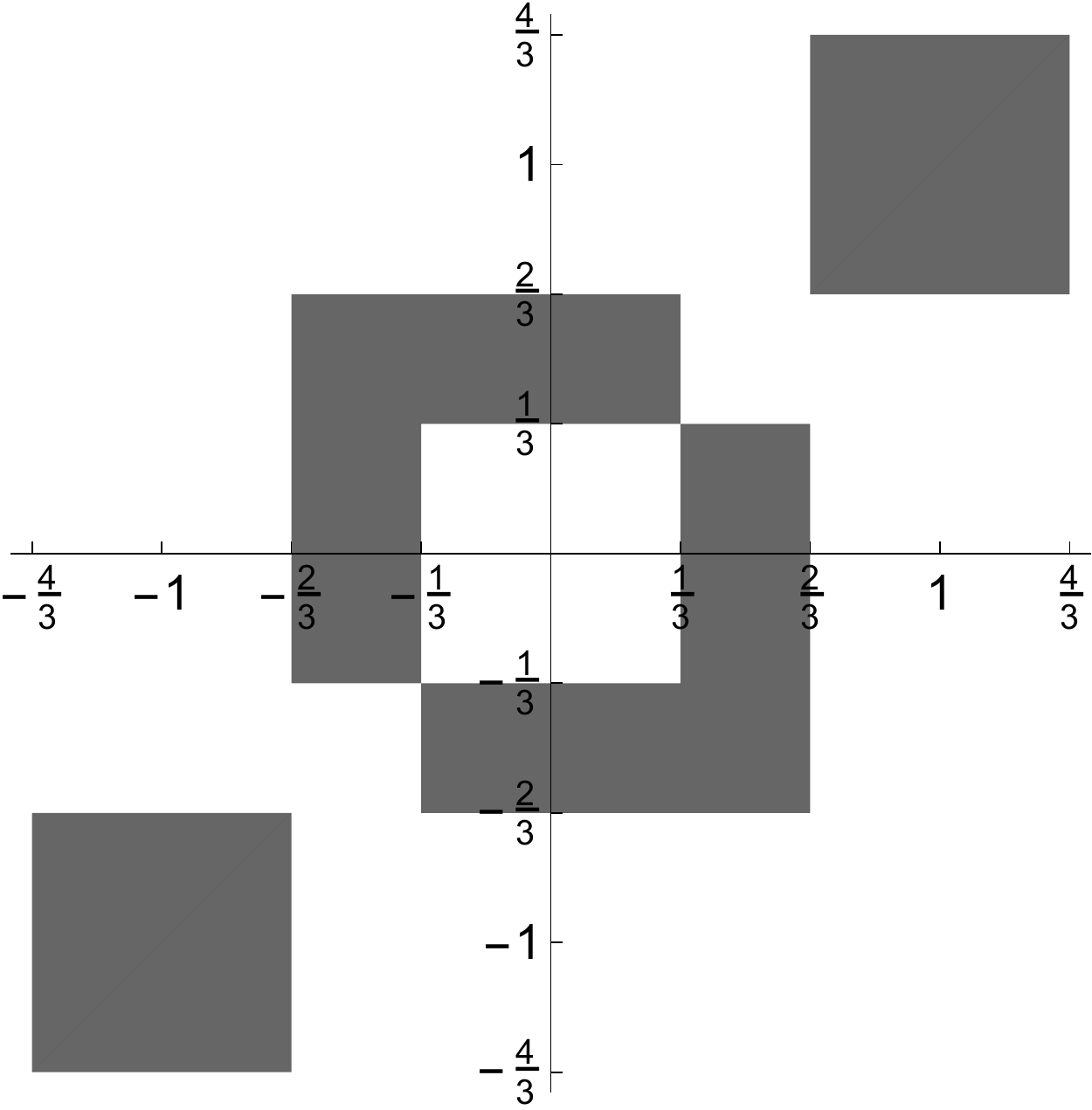}}
\end{picture}
\end{array}$
\label{pgwave}
\caption{Conventional 2-wavelet sets for dilation by 2.}
\label{2wave}
\end{figure}

\par For pg and pm, which we know have the same wavelet sets by Corollary \ref{glide}, we divide the diamond annulus 2-wavelet set on the left in Figure \ref{2wave} into two disjoint pieces, each of which tiles under translation, such that one is the vertical reflection of the other.  Specifically,   
let $W^{\rm pm}={\rm conv}\{(0,\frac13), (0,\frac23), (2,0), (1,0)\}\cup {\rm conv}\{(1,0), (2,0), (\frac12,-\frac12),(0,-\frac13)\}\cup {\rm conv}\{(0,-\frac13),(0,-\frac23),(-\frac12,-\frac12)\}$, as shown on the left in Figure \ref{pgwave}.  For p2, we start with the same 2-wavelet set, and this time divide it into two pieces such that each tiles under translation and one is rotation by $\pi$ of the other.  The resulting wavelet set $W^{\rm p2}$, which is formed by translating the lower left triangle of $W^{\rm pm}$ by $(0,1)$, is shown in the center of Figure \ref{pgwave}. Tiling under translation can be established for each by piecewise integer translating them into $[0,1]\times[-\frac13,\frac23]$. The other conditions of Theorem \ref{wscond} follow from the fact that the lefthand 2-wavelet set in Figure \ref{2wave} is the disjoint union of each of these sets acted on by their point groups.\\     
 \begin{figure}[h]
\centering $\begin{array}{ccc}
 \setlength{\unitlength}{115bp}
\begin{picture}(1,1)(0,-.2)
\put(0,0){\includegraphics[width=\unitlength]{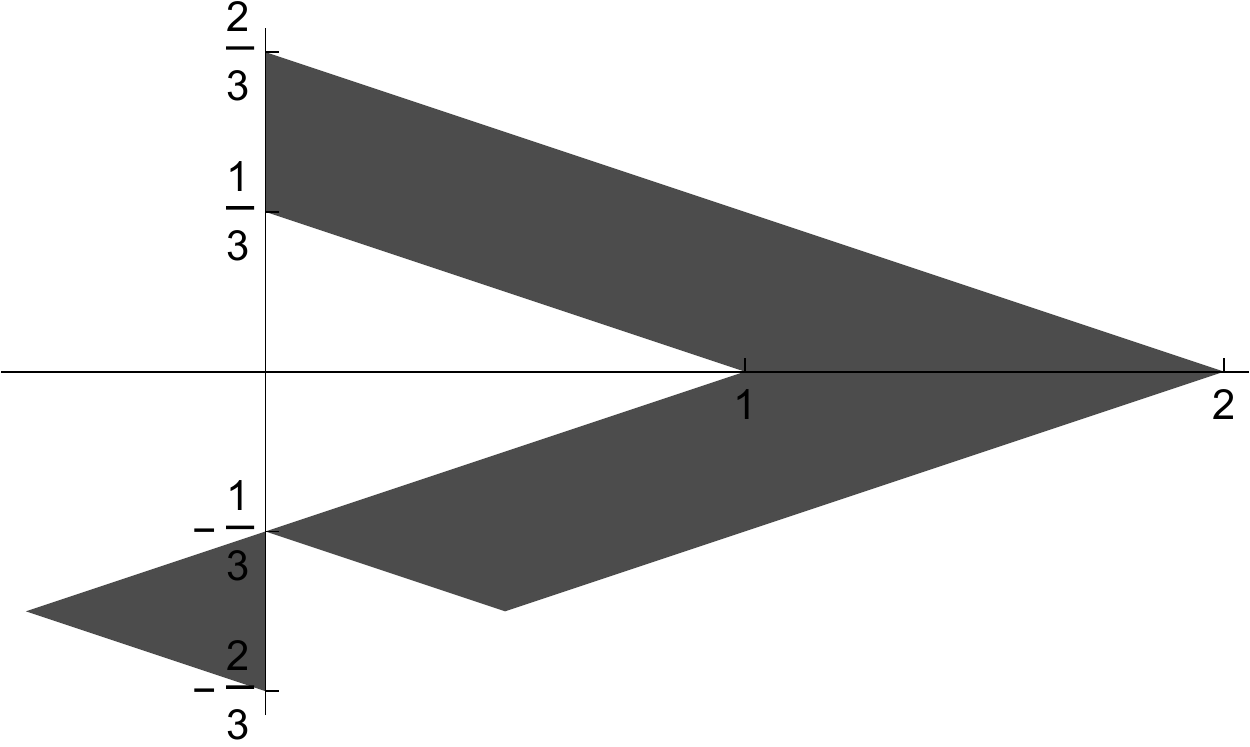}}
\end{picture}&
\setlength{\unitlength}{115bp}
\begin{picture}(1,1)(.05,-.3)
\put(0,0){\includegraphics[width=\unitlength]{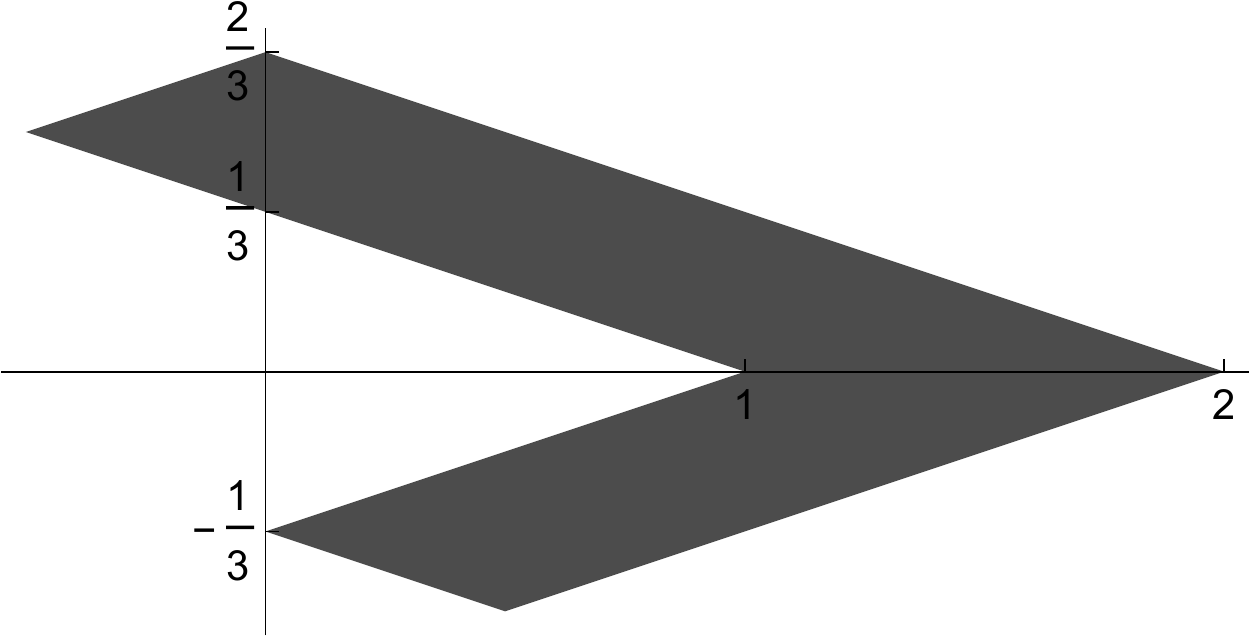}}
\end{picture}&
\setlength{\unitlength}{95bp}
\begin{picture}(1,1)(.05,-.1)
\put(0,0){\includegraphics[width=\unitlength]{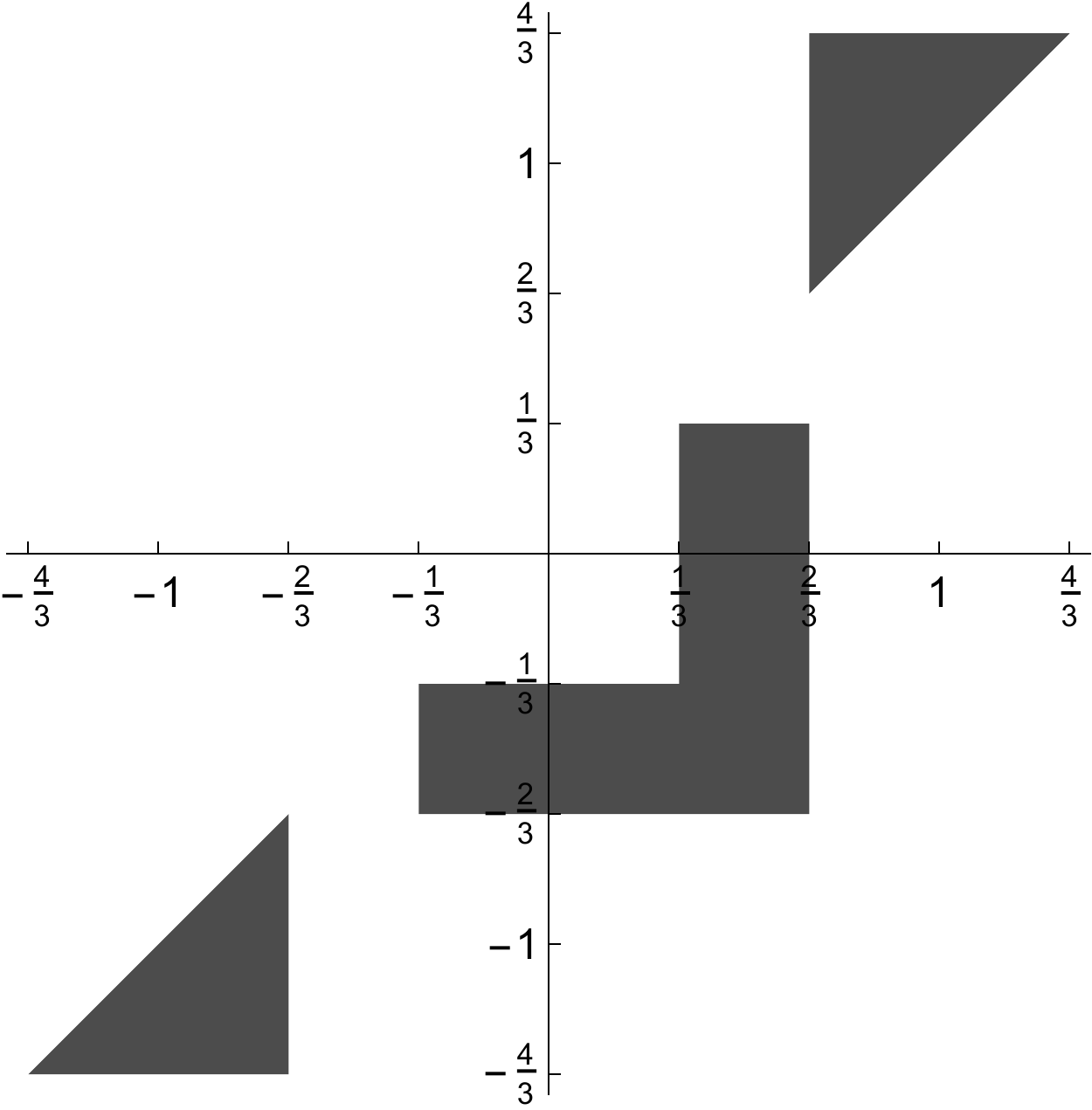}}
\end{picture}\\
W^{\rm pm}&W^{\rm p2}&W^{\rm cm}
\end{array}$
\caption{Simple wavelet sets\index{simple wavelet set} for crystallographic groups with order 2 point groups.}
\label{pgwave}
\end{figure}

Our ability to break this conventional 2-wavelet set into two eligible pieces for the groups pm, pg and p2 required it to be symmetric with respect to the reflection and rotation in those point groups.  For cm, we use instead the conventional 2-wavelet set shown on the right in Figure \ref{2wave}, which is symmetric with respect to $\sigma_{(1,1)}$.  We divide this 2-wavelet set into two disjoint pieces such that each is the reflection of the other, choosing so that each piece tiles\index{tiling} under integer translation.  Specifically, let $W^{\rm cm}={\rm conv}\{(\frac23,\frac23), (\frac23,\frac43), (\frac43,\frac43)\}\cup{\rm conv}\{(-\frac23,-\frac23), (-\frac23,-\frac43), (-\frac43,-\frac43)\}\;\cup\;{\rm conv}\{(-\frac13,-\frac13), (-\frac13,-\frac23), (\frac23,-\frac23),(\frac23,\frac13), (\frac13,\frac13), (\frac13,-\frac13),(-\frac13,-\frac13)\}$.  By Theorem \ref{wscond}, $W^{\rm cm}$, shown on the right in Figure \ref{pgwave},  is a wavelet set for cm.  

\newsavebox{\double}
\savebox{\double}{ $\left(\begin{smallmatrix}
1&0\\
0&2
\end{smallmatrix}\right)$}

We now alter the conventional 2-wavelet set of Figure \ref{2wave} to create a conventional 4-wavelet set that can be used for four of the five crystallographic groups whose point groups have order 4.  To do this, we apply the matrix \usebox{\double} to the 2-wavelet set.  This converts the set of Figure \ref{2wave} into a rhombic annulus that tiles under translation by the lattice spanned by $(1,0)$ and $(0,2)$.   As this new set, shown in the center image of Figure \ref{p4ws}, still tiles under dilation by 2, and 4-tiles under translation by the integers, it is a conventional 4-wavelet for this dilation.  

\par Break this rhombic annulus into four translation tiling pieces that are reflections of one another in the x and y axes, as shown in the center image of Figure \ref{p4ws}, to give a single wavelet set for pmm, pmg, and pgg.  The wavelet set $W^{\rm pmm}$, which is shown on the left of Figure \ref{p4ws}, is of the form ${\rm conv}\{(0,\frac23),(0,1),(\frac12,1),(2,0),(1,0)\}\cup{\rm conv}\{(-\frac12,-1),(0,-1),(0-\frac43)\}$.  To see that $W^{\rm pmm}$ tiles under integer translation, piecewise integer translate it  into $[0,1]^2$.  The set $W^{\rm pmm}$ is thus a wavelet set for pmm, pmg, and pgg since the rhombic 4-wavelet set is the disjoint union of their point group acting on $W^{\rm pmm}$.  The right image in Figure \ref{p4ws} shows $\cup_{S\in\mathcal B^{\rm p4}}S(W^{\rm pmm})$ for p4, which is also a disjoint union that tiles under dilation by 2.  Thus $W^{\rm pmm}$ is also a wavelet set for p4.   The conventional 4-wavelet  set in the center of Figure \ref{p4ws} does not have enough symmetry to give a simple wavelet set for cmm, the final crystallographic group with point group of order 4.  We will present one derived by a different method later in this section.
 
 \begin{figure}[h]
\centering 
$\begin{array}{ccc}
 \setlength{\unitlength}{90bp}
\begin{picture}(1,1)(0,-.1)
\put(0,0){\includegraphics[width=\unitlength]{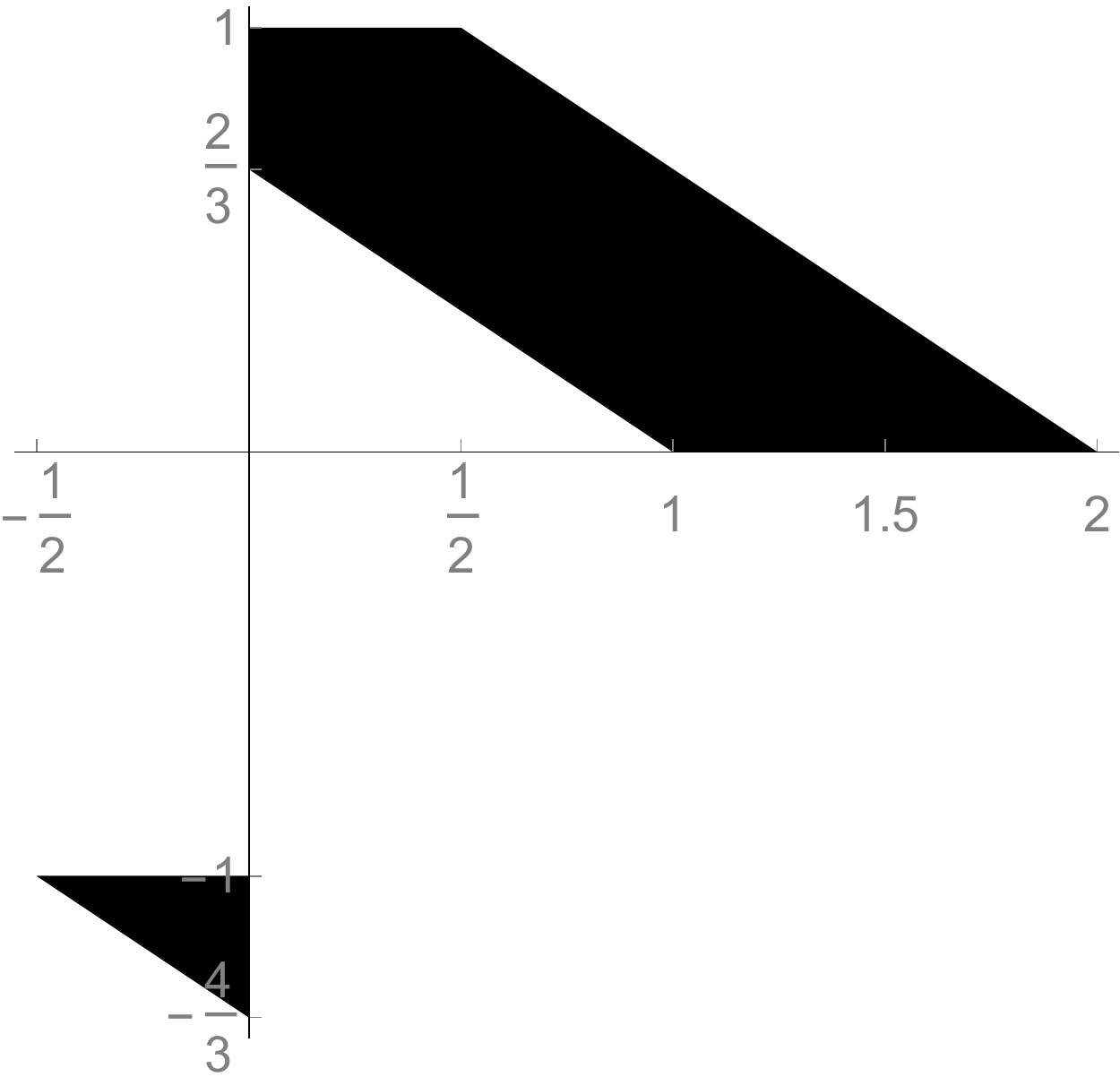}}
\end{picture}&
 \setlength{\unitlength}{120bp}
\begin{picture}(1,1)(0,-.1)
\put(0,0){\includegraphics[width=\unitlength]{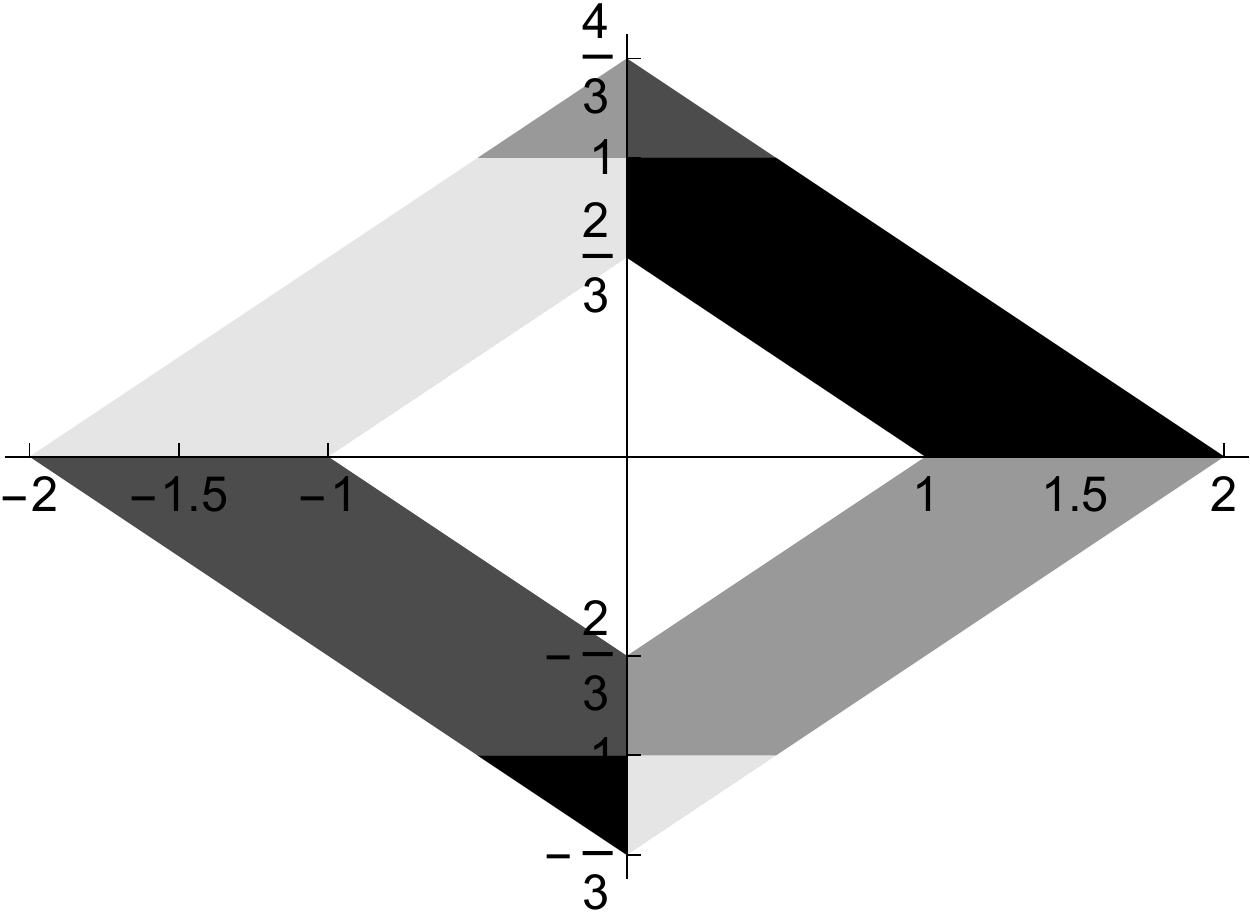}}
\end{picture}&
\setlength{\unitlength}{120bp}
\begin{picture}(1,1)(0,.02)
\put(0,0){\includegraphics[width=\unitlength]{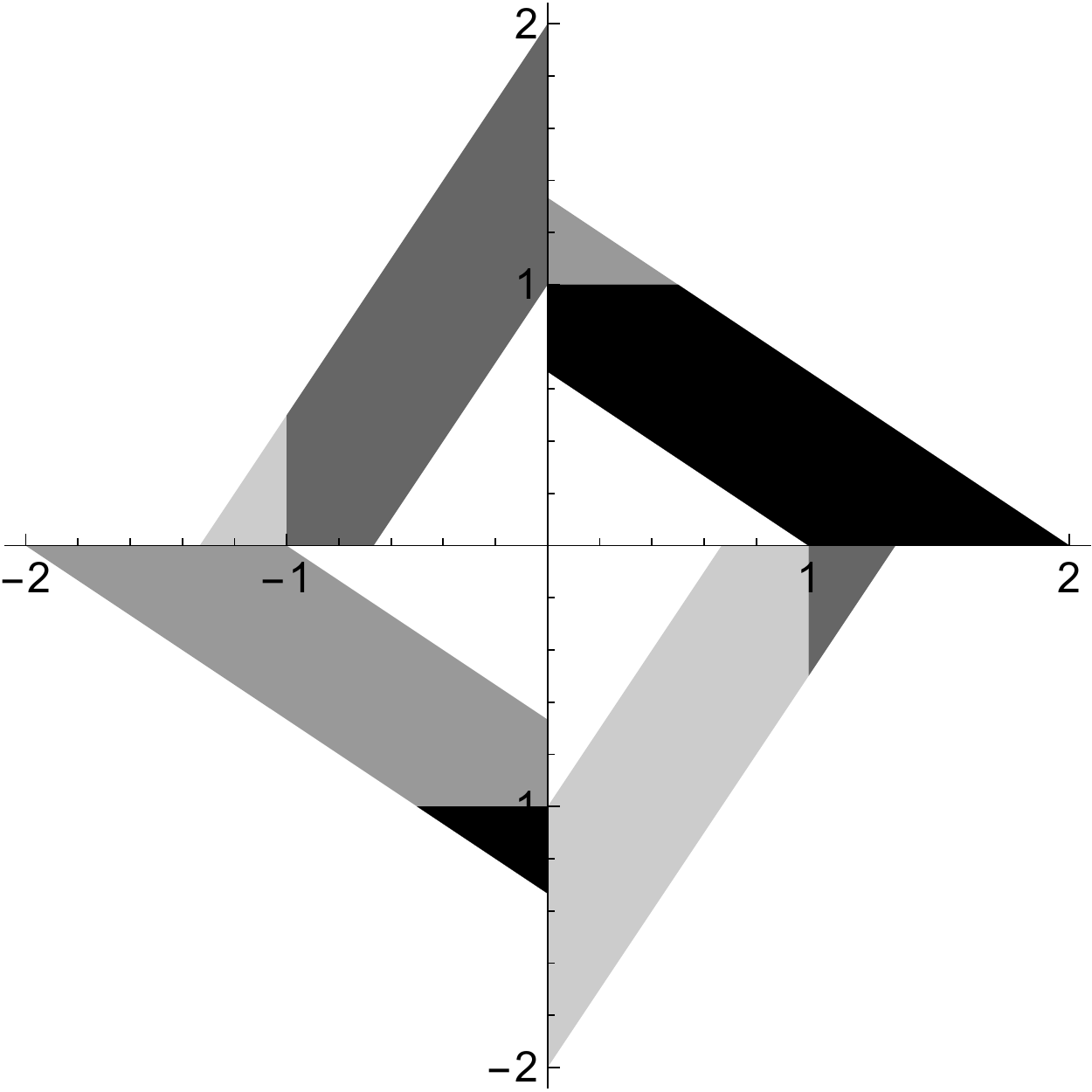}}
\end{picture}\\
W^{\rm pmm}&\cup_{S\in\mathcal B^{\rm pmm}}S(W^{\rm pmm})&\cup_{S\in\mathcal B^{\rm p4}}S(W^{\rm pmm})
\end{array}$
\caption{Simple wavelet set for pmm, pmg, pgg and p4\index{simple wavelet set}}
\label{p4ws}
\end{figure}  

\par 
As in the previous section, the simple wavelet set for the group p4 can be multiplied by the change of lattice\index{lattice} $L=\usebox{\rhomat}$ to give a wavelet set for p6.  The set $W^{\rm p6}=LW^{\rm pmm}$ is shown on the left in Figure \ref{p6wave}, along with its rotations under powers of $\rho_6$ on the right.  The fact that $W^{\rm p6}$ tiles under translation by $\mathcal L$ again follows immediately from the fact that $W^{\rm pmm}$ tiles under translation by $\mathbb Z^2$.  The other two requirements of Theorem \ref{wscond} follow by noting that multiplication by $L$ sends the polar sector $\theta\in[0,\frac{\pi}2]$ to the polar sector $\theta\in[0,\frac {\pi}3]$, and preserves rotation by $\pi$ and dilation by 3.  Unlike in Section \ref{cws}, multiplying $W^{\rm pmm}$ by $L'$ does not produce a wavelet set for p3.  This is because $W^{\rm pmm}$ has a piece that has been rotated by $\pi$ from its first quadrant position, and p3 is not compatible with rotation by $\pi$.   
\begin{figure}[h]
\centering $\begin{array}{cc}
 \setlength{\unitlength}{100bp}
\begin{picture}(1,1)(0,-.3)
\put(0,0){\includegraphics[width=\unitlength]{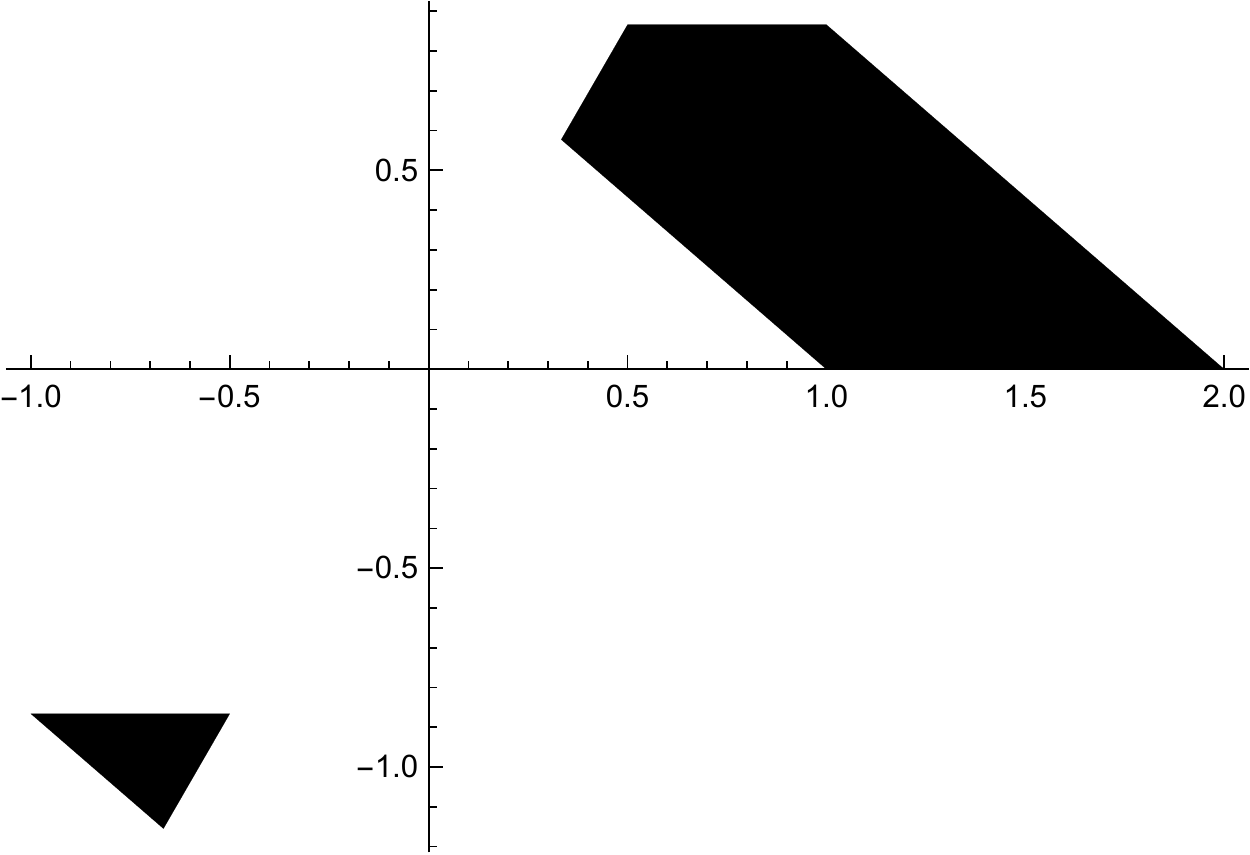}}
\end{picture}&
\setlength{\unitlength}{120bp}
\begin{picture}(1,1)(-.15,-.15)
\put(0,0){\includegraphics[width=\unitlength]{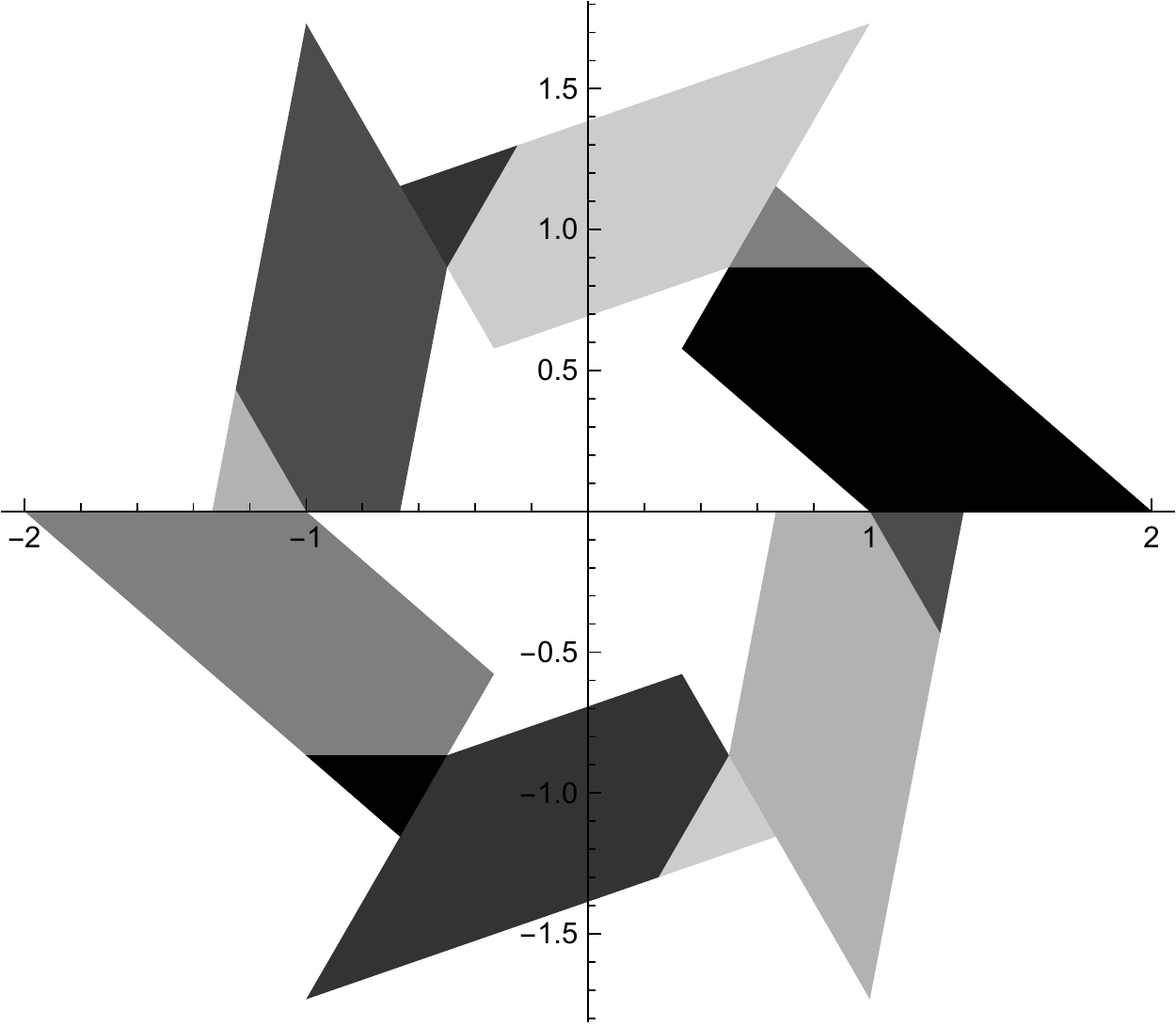}}
\end{picture}
\end{array}$
\caption{Simple wavelet set $W^{\rm p6}=LW^{\rm pmm}$ for p6,  and its 6-fold rotations.}
\label{p6wave}
\end{figure}  
\par To find a simple wavelet set for p3 and also for rest of the remaining seven crystallographic groups, we turn to the conventional first quadrant subspace wavelet set provided by the following lemma. 
\begin{lemma}\label{4square} Let $A=\cup_{j=0}^2 [\frac{2j}3,\frac{2(j+1)}3)]^2$.  Then the set 
$A\setminus \frac12 A$ gives a conventional subspace wavelet set for the first quadrant under dilation by 2.\end{lemma}
 \begin{proof}
Since $\frac12 A\subset A$, the set $A\setminus \frac12 A$ tiles\index{tiling} under dilation by 2.  Note that $A\setminus \frac12 A$ has area equal to $1$.  Thus, using a classical result from harmonic analysis (see e.g.\cite{kou}), we can establish tiling under integer translation by showing that $\widehat{\bf 1}_{A\setminus\frac12 A}$ vanishes on $\mathbb Z^2\setminus {0}$. Assume $k,\ell$ are both nonzero.  (The argument for one nonzero is similar.)  We have
 \begin{equation}
 \widehat{\bf 1}_A(k,\ell)=\frac{e^{-\frac23i\pi(k+\ell)}\sin(\frac{2k\pi}3)\sin(\frac{2\ell\pi}3)}{k\ell\pi^2}\sum_{j=0}^2 e^{\frac{-2\pi i 2j(k+\ell)}3}
 \end{equation}
 The sum at the end of this expression is 0 except when $k+\ell\in3\mathbb Z.$  The formula for $\widehat{\bf 1}_{\frac12 A}(k,\ell)$ is similar (with $\frac23$ replaced by $\frac13$), and thus is also 0 except when $k+\ell\in3\mathbb Z.$  Thus, it remains to address the case  $k+\ell\in3\mathbb Z.$
In this case, write $\ell=3n-k$ and note that $\widehat{\bf 1}_A(k,\ell)-\widehat{\bf 1}_{\frac12 A}(k,\ell)$ has a factor of  
 $\sin(\frac{2k\pi}3)\sin(\frac{2(3n-k)\pi}3)-e^{-i\pi n}\sin(\frac{k\pi}3)\sin(\frac{(3n-k)\pi}3)$, which is 0.
 \end{proof}
 
This conventional first quadrant subspace wavelet set $W^{\rm p4}$, shown on the left in Figure \ref{p3ws}, gives an alternative simple crystallographic wavelet set for p4, pmm, pmg, and pgg.  (To differentiate it from the previous simple wavelet set for these groups, we name it $W^{\rm p4}$.)  Unlike $W^{\rm pmm}$, it can be multiplied by the matrix $L'$, as in the previous section, to give a simple wavelet set for p3.  $W^{\rm p3}$ is shown in the center image of Figure \ref{p3ws}, and together with its rotates on the right.  We will see that variations of the set $W^{\rm p4}$ can also be used to create simple dilation 2 wavelet sets for all the remaining crystallographic groups.

 \begin{figure}[h]
\centering 
$\begin{array}{ccc}
 \setlength{\unitlength}{80bp}
\begin{picture}(1,1)(.35,0)
\put(0,0){\includegraphics[width=\unitlength]{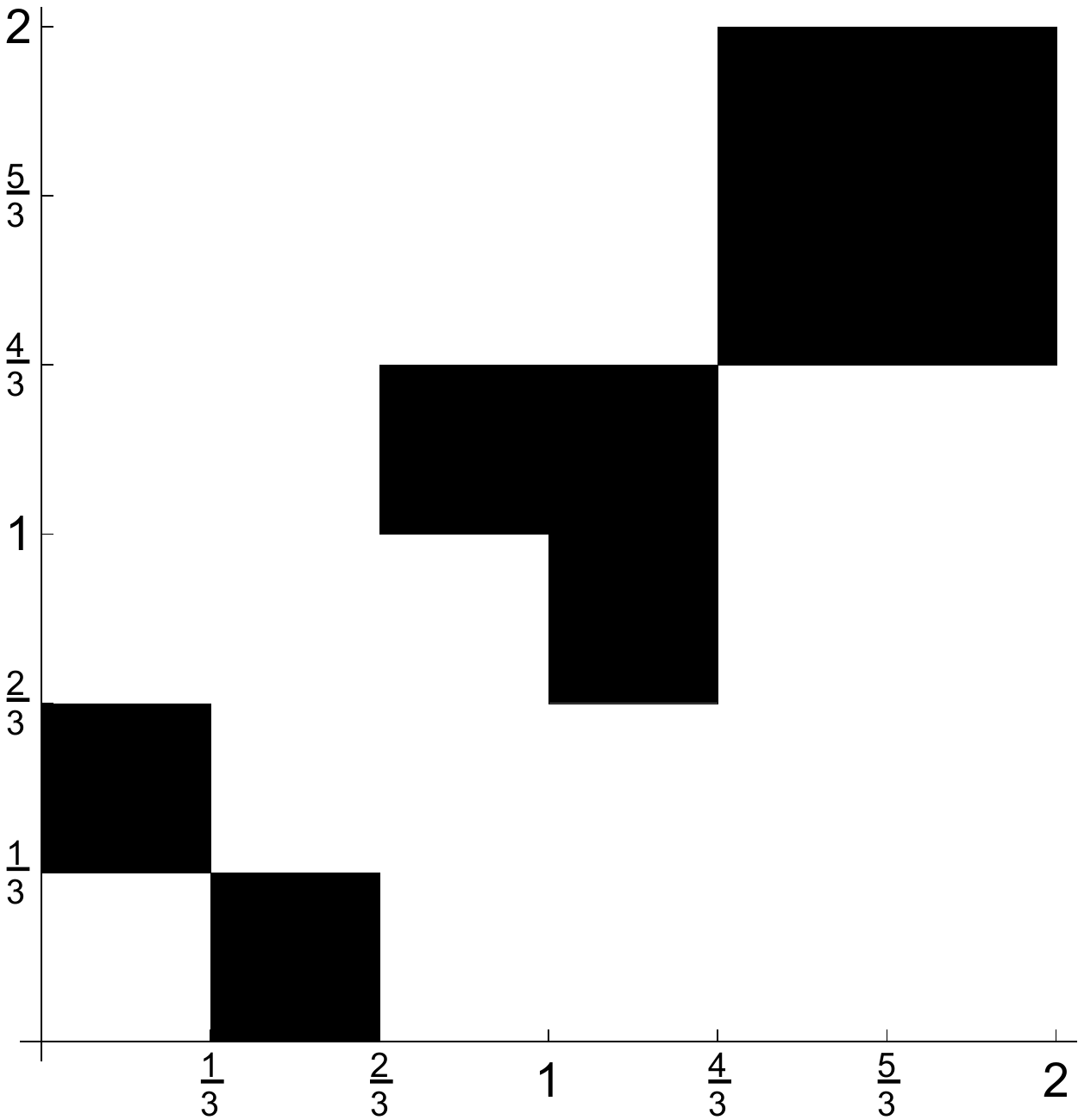}}
\end{picture}&
 \setlength{\unitlength}{70bp}
\begin{picture}(1,1)(-.2,-.1)
\put(0,0){\includegraphics[width=\unitlength]{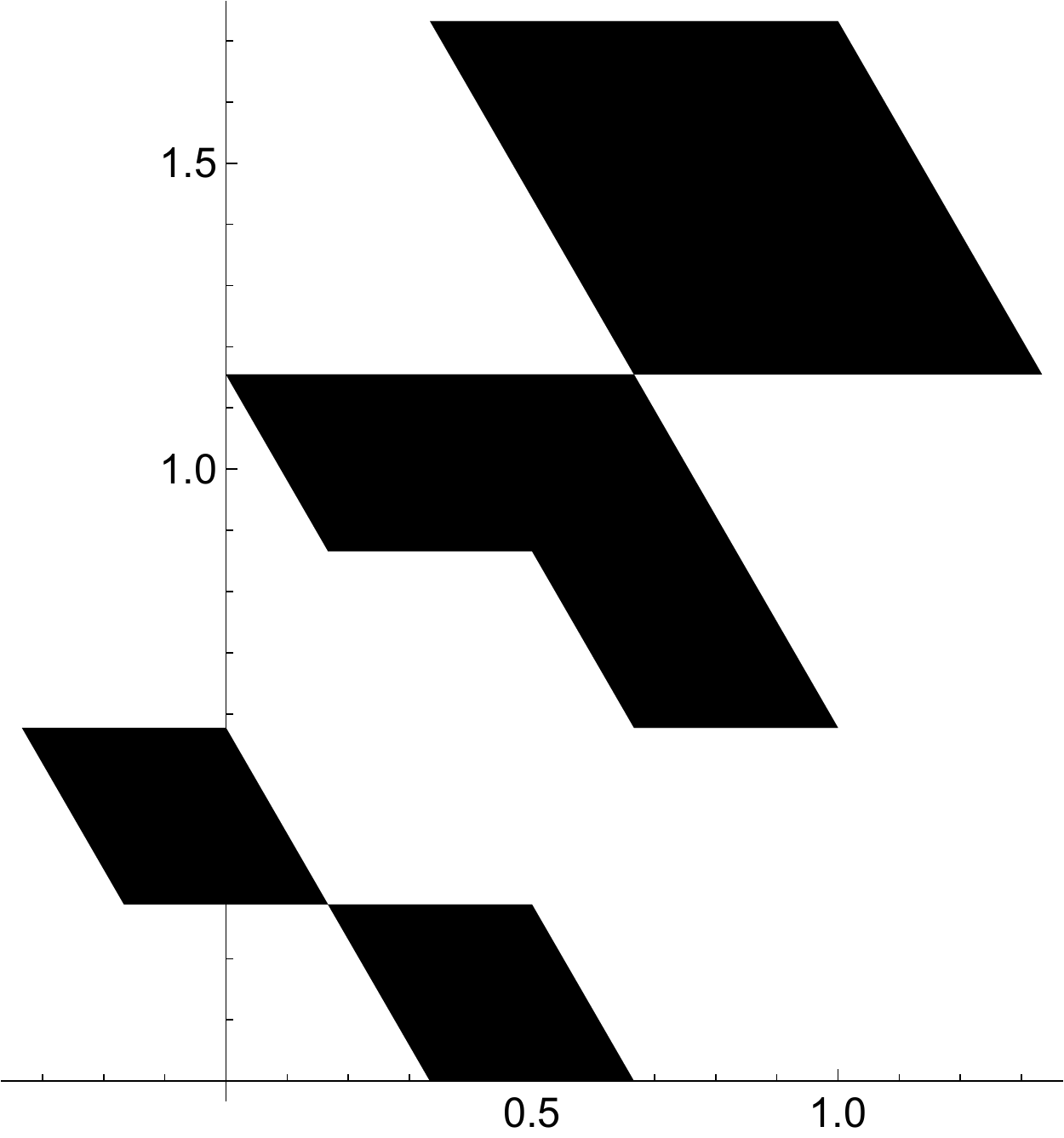}}
\end{picture}&
\setlength{\unitlength}{90bp}
\begin{picture}(1,1)(-.35,.0)
\put(0,0){\includegraphics[width=\unitlength]{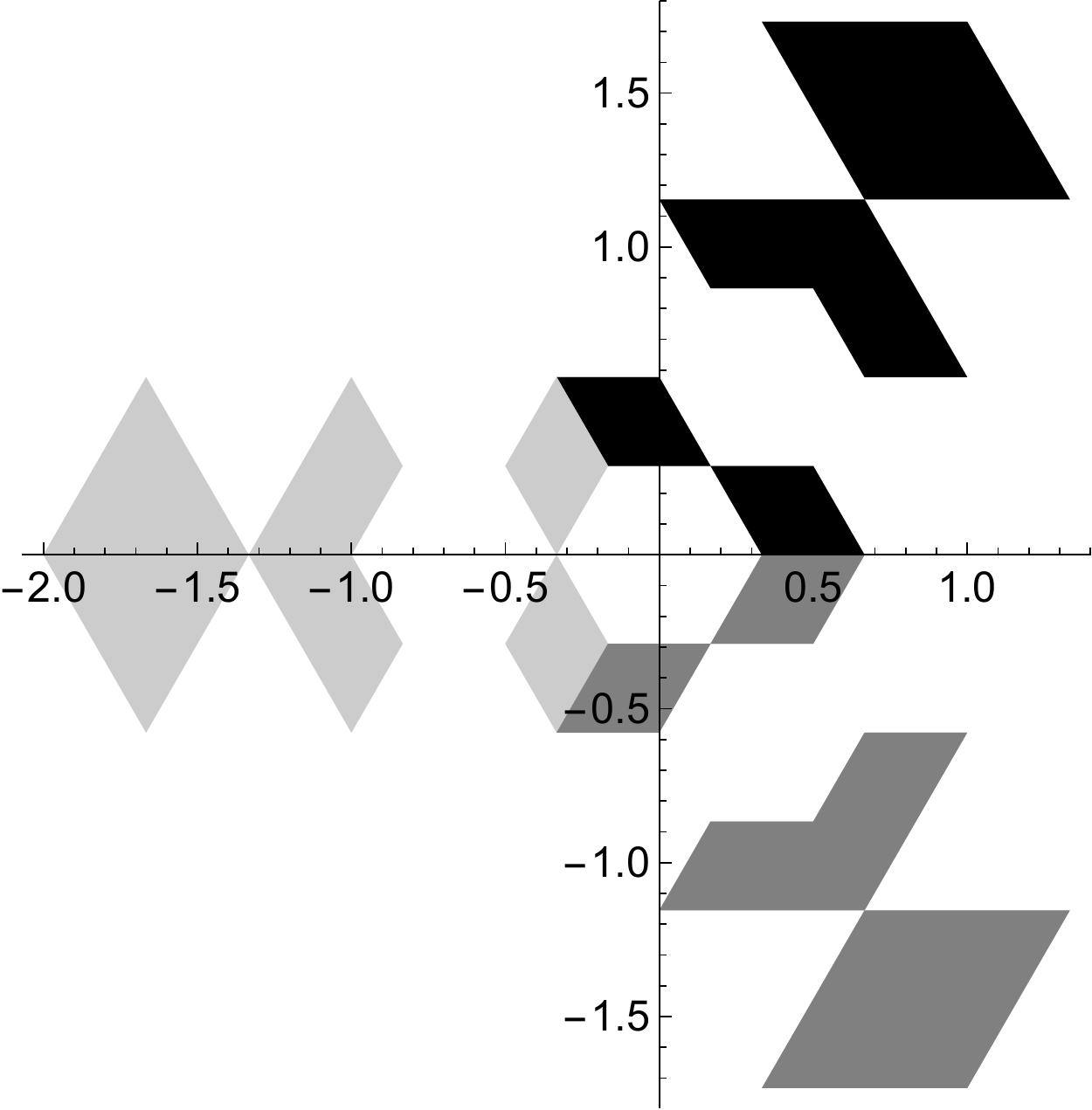}}
\end{picture}\\
W^{\rm p4}&W^{\rm p3}=L'(W^{\rm p4})&\hskip.7in\cup_{S\in\mathcal B^{\rm p3}}S(W^{\rm p3})
\end{array}$
\caption{Alternative simple wavelet set $W^{\rm p4}$ for pmm, pmg, pgg, and p4; and simple wavelet set for p3 together with its rotations.\index{simple wavelet set}}
\label{p3ws}
\end{figure} 

First, to find a wavelet set for the group cmm, we divide $W^{\rm p4}$ into two sets: the portion $W^{\rm p4}_a$, which has integer translates that fill the upper and lower quarter triangles of the unit square $[0,1]^2$ (shown in black on the left in Figure \ref{cmmws}), and the portion $W^{\rm p4}_b$, which has integer translates that fill the left and right quarter triangles (shown in gray).  We then reflect $W^{\rm p4}_b$ over the $x-$axis into the fourth quadrant.  Since the $\sigma_{(1,0)}W^{\rm p4}_b$ also has integer translates filling the left and right quarter triangles, the set $W^{\rm cmm}=W^{\rm p4}_a\cup\sigma_{(1,0)}W^{\rm p4}_b$, shown in the center image of Figure \ref{cmmws}, tiles\index{tiling} under translation by the integers.  The right image of Figure \ref{cmmws} shows the disjoint union of $W^{\rm cmm}$ and its reflections in the diagonals, which tiles the plane under dilation by 2.  Thus $W^{\rm cmm}$ is a simple wavelet set for cmm.     
\begin{figure}[h]
\centering $\begin{array}{ccc}
 \setlength{\unitlength}{75bp}
\begin{picture}(1,1)(0.3,-.2)
\put(0,0){\includegraphics[width=\unitlength]{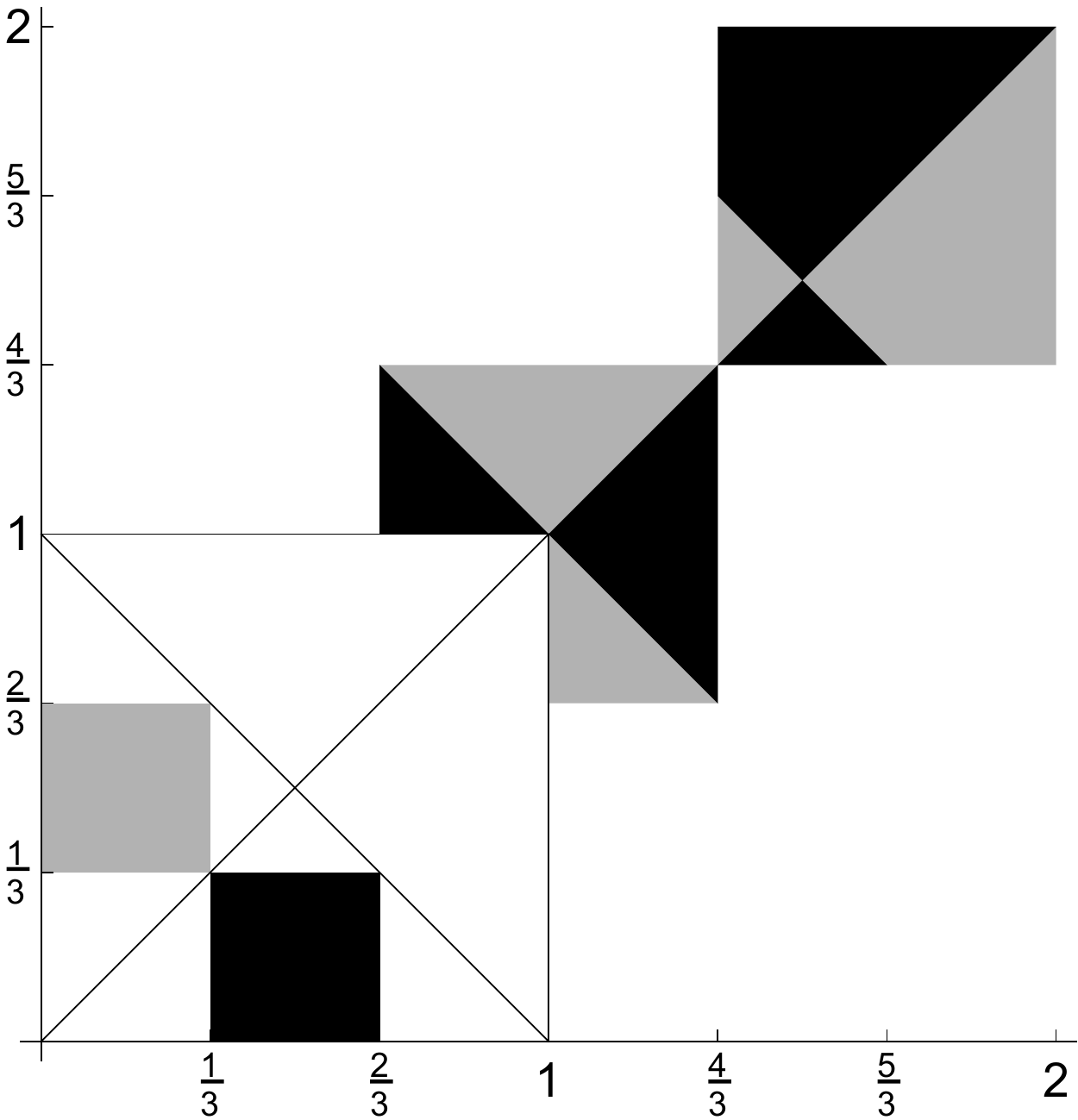}}
\end{picture}&
\setlength{\unitlength}{60bp}
\begin{picture}(1,1)(.3,0)
\put(0,0){\includegraphics[width=\unitlength]{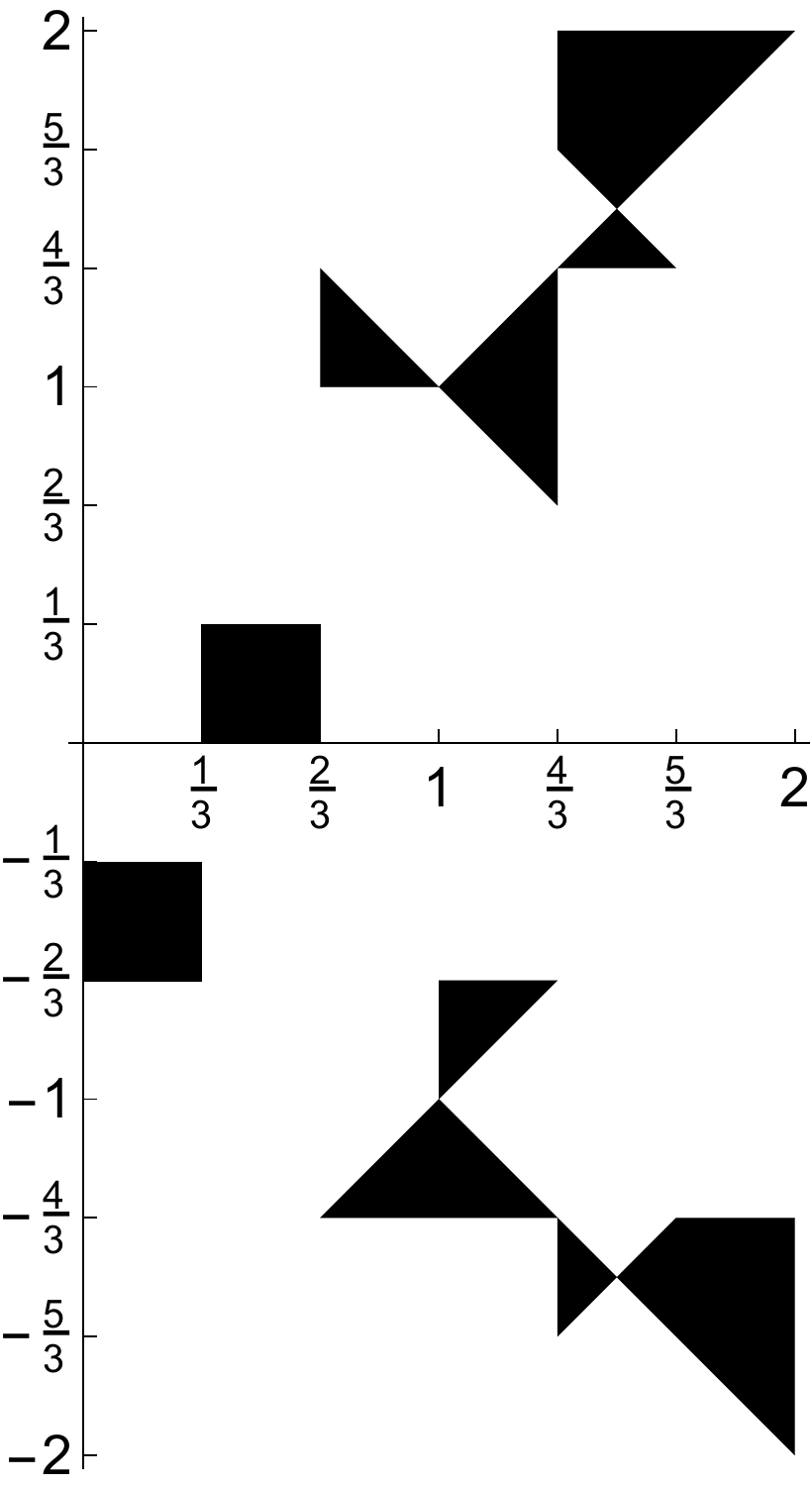}}
\end{picture}&
\setlength{\unitlength}{100bp}
\begin{picture}(1,1)(0,-.1)
\put(0,0){\includegraphics[width=\unitlength]{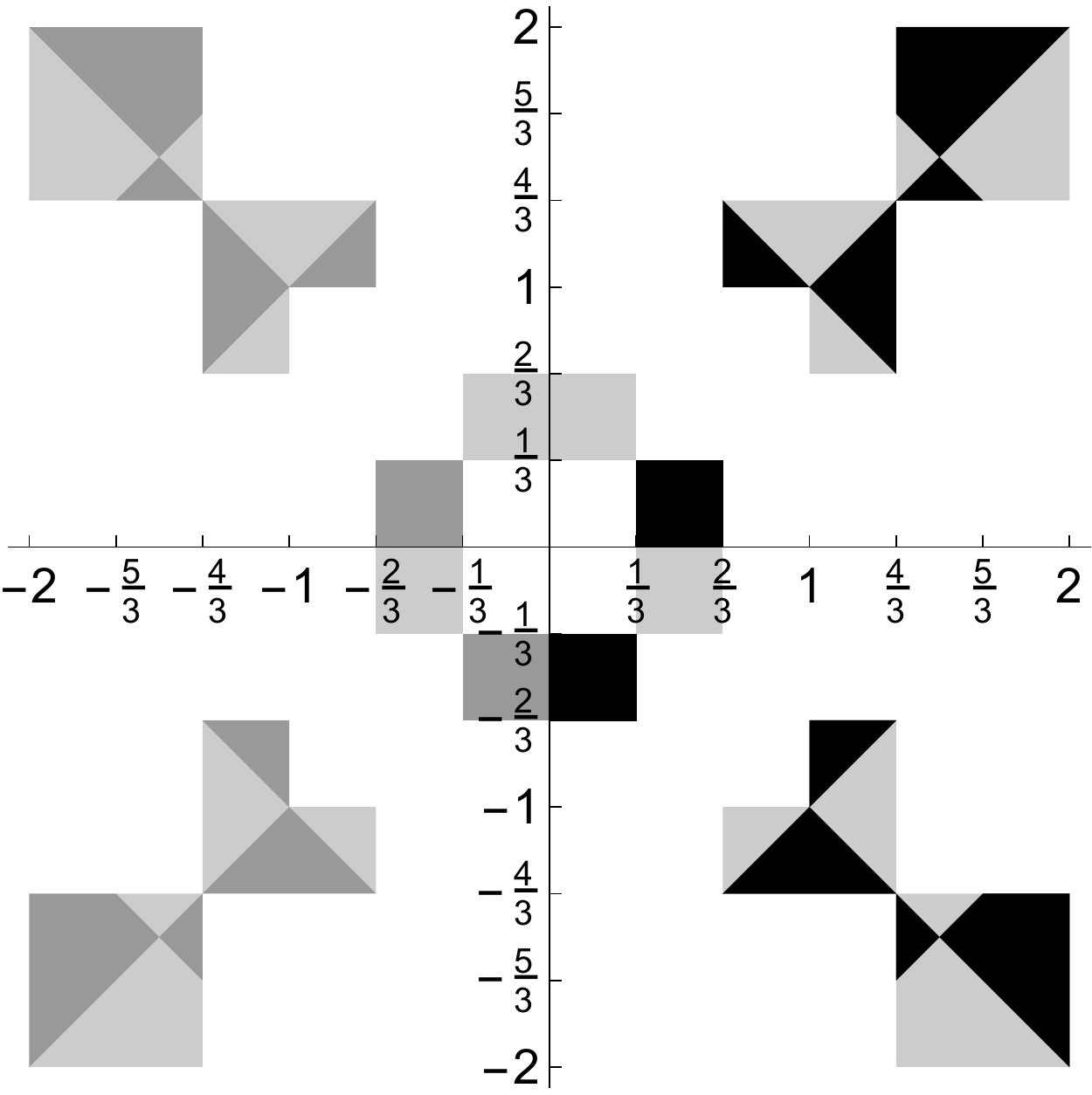}}
\end{picture}\\
W^{\rm p4}=W^{\rm p4}_a\cup W^{\rm p4}_b\hskip.2in&W^{\rm cmm}=W^{\rm p4}_a\cup \sigma_{(1,0)}W^{\rm p4}_b\hskip.2in&\cup_{S\in\mathcal B^{\rm cmm}} SW^{\rm cmm} \hskip.2in
\end{array}$
\caption{Building a simple wavelet set for for cmm\index{simple wavelet set}.}
\label{cmmws}
\end{figure}     

Next, we will build wavelet sets for crystallographic groups whose point groups have order 8 out of a conventional subspace 2-wavelet set for the first quadrant.  For this set, we use two copies of $W^{\rm p4}$ along the main diagonal, slightly altered by an integer translation of three of the smaller squares.    Specifically, consider the set $C=\left(A\cup\tau_{(2,2)}A\right)\setminus\frac12\left(A\cup\tau_{(2,2)}A\right)$, where $A$ is as in Lemma \ref{4square} . The set $C$ gives a two fold tiling of the first quadrant under integer translation since it can be formed from integer translations of pieces of two copies of $W^{\rm p4}$.  Since $\frac12\left(A\cup\tau_{(2,2)}A\right)\subset\left(A\cup\tau_{(2,2)}A\right)$, it also tiles under dilation by 2.  This set appears in the first quadrant of the center image in Figure \ref{p4gwave}.  

The simple wavelet set for p4m and p4g is formed by taking $W^{\rm p4m}$ to be a set that tiles the first quadrant under integer translation and has the disjoint union $W^{\rm p4m}\cup\sigma_{(1,1)}W^{\rm p4m}=C$.  Specifically, let  $W^{\rm p4m}={\rm conv}\{(0,\frac13), (0,\frac23), (\frac13,\frac23), (\frac13,\frac13)\}\;\cup\;{\rm conv}\{(1,\frac23),(1,1), (\frac43,1),(\frac43,\frac23)\}\;
 \cup\;{\rm conv}\{(\frac43,2), (\frac53,2), (\frac53,\frac53), (\frac43,\frac53)\}\;\cup\;{\rm conv}\{(2,2),\\ (\frac83,2), (\frac83,\frac83)\}\;\cup {\rm conv}\{(\frac83,\frac83),(\frac83,\frac{10}3), (\frac{10}3,\frac{10}3)\}\;\cup\;{\rm conv}\{(\frac{10}3,\frac{10}3),(\frac{10}3,\frac{11}3), (\frac{11}3,\frac{11}3)\}
\; \cup\\
{\rm conv}\{(\frac{11}3,\frac {11}3),(4,4), (4,\frac{10}3),(\frac{11}3,\frac{10}3)\}$.  (See set on the left in Figure \ref{p4gwave}.)  Tiling under integer translation can be established by piecewise integer translating $W^{\rm p4m}$ into $[0,1]^2$.   
The effect of $\sigma_{(1,1)}$ on $W^{\rm p4m}$ is shown in dark gray in the center graph of Figure \ref{p4gwave}, and then the  combined effect of this reflection and the 4-fold rotations in the lighter shades.  A wavelet set for p6m can be found by mapping $W^{\rm p4m}$ by the lattice matrix $L$.  The result, shown on the right in Figure \ref{p4gwave},  meets all the requirements of Theorem \ref{wscond} by a similar argument as used for p6.  
\newsavebox{\sigmadmat}
\savebox{\sigmadmat}{ $\left(\begin{smallmatrix}
0&-1\\
-1&0
\end{smallmatrix}\right)$}
  \begin{figure}[h]
\centering $\begin{array}{ccc}
 \setlength{\unitlength}{90bp}
\begin{picture}(1,1)(.1,-.1)
\put(0,0){\includegraphics[width=\unitlength]{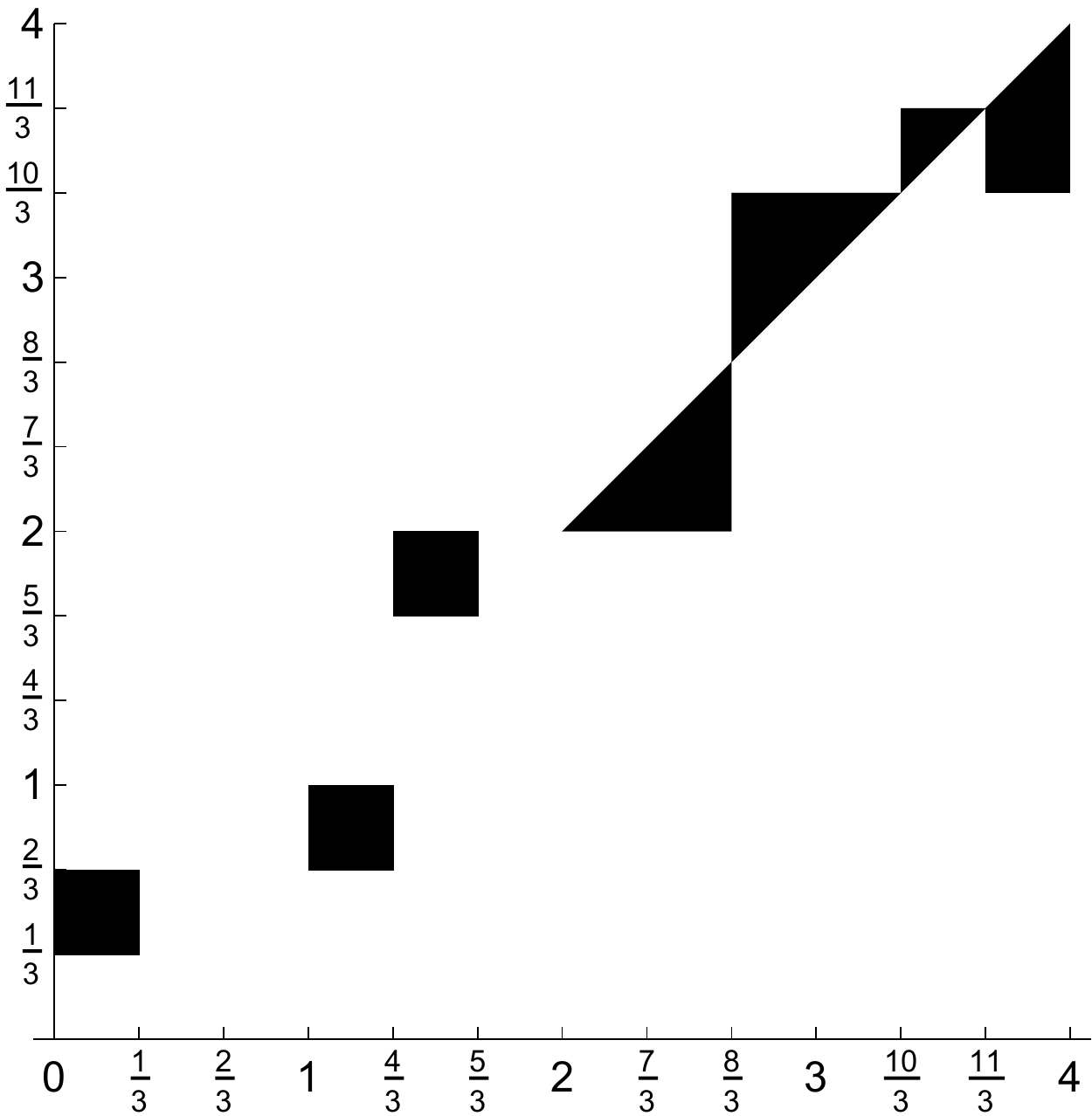}}
\end{picture}&
\setlength{\unitlength}{120bp}
\begin{picture}(1,1)(0,0)
\put(0,0){\includegraphics[width=\unitlength]{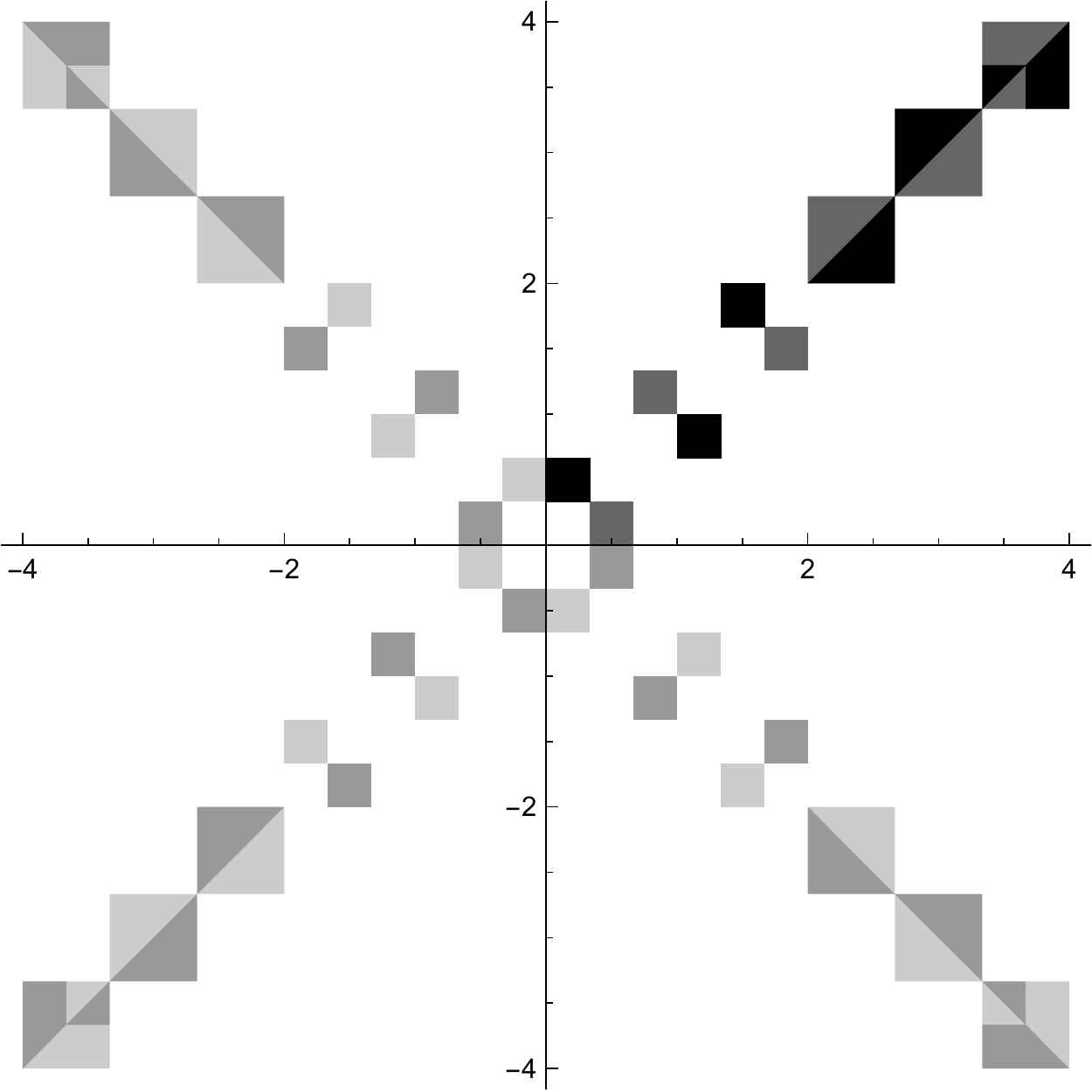}}
\end{picture}&
\setlength{\unitlength}{100bp}
\begin{picture}(1,1)(-.1,-.3)
\put(0,0){\includegraphics[width=\unitlength]{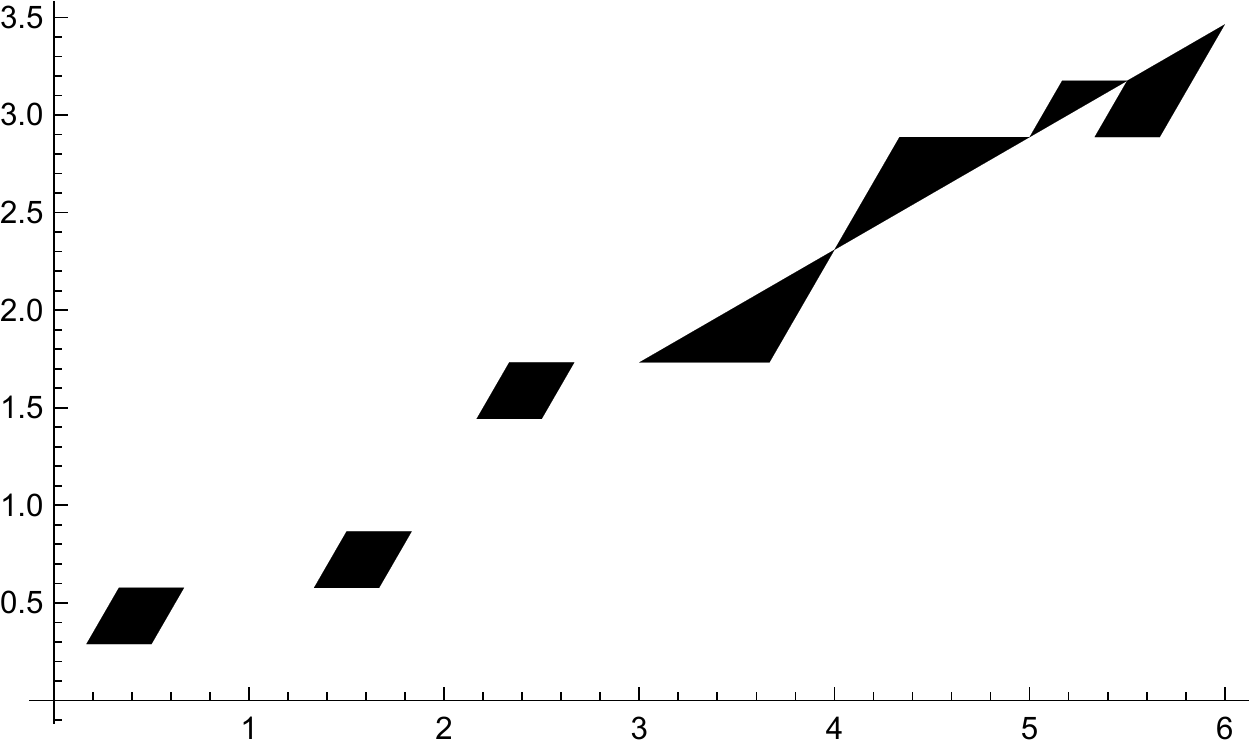}}
\end{picture}\\
W^{\rm p4m}&\cup_{S\in\mathcal B^{\rm p4m}} SW^{\rm p4m}&W^{\rm p6m}=LW^{\rm p4m}
\end{array}$
\caption{Simple wavelet set for p4m and p4g, acted on by point group, and simple wavelet set for p6m\index{simple wavelet set}}
\label{p4gwave}
\end{figure}
\par The wavelet set for p4m can also be multiplied by the lattice\index{lattice} matrix $L'$ to give one for p31m.  Instead, we use a subspace wavelet set for the first and fifth octants to produce a wavelet set for both p31m and p3m1.  To get the subspace wavelet set, we modify the first quadrant conventional subspace wavelet set $W^{\rm p4}$ by deleting the portion in the second octant, and including instead a copy of the first octant rotated by $\pi$. The resulting set $D$, shown on the left in Figure \ref{p31md2}, is easily seen to tile\index{tiling} the plane under integer translation and the first and fifth octants by dilation by 2, and so is a conventional subspace wavelet set for those two octants.  The set $LD$, shown in the center image of  Figure \ref{p31md2} is then a wavelet set for p3m1 and also for p31m. The graph on the right in Figure \ref{p31md2} shows the union of the wavelet set with its reflection and rotations, which dilates to a hexagonal star annulus.  

\begin{figure}[h]
\centering $\begin{array}{ccc}
 \setlength{\unitlength}{95bp}
\begin{picture}(1,1)(0,-.1)
\put(0,0){\includegraphics[width=\unitlength]{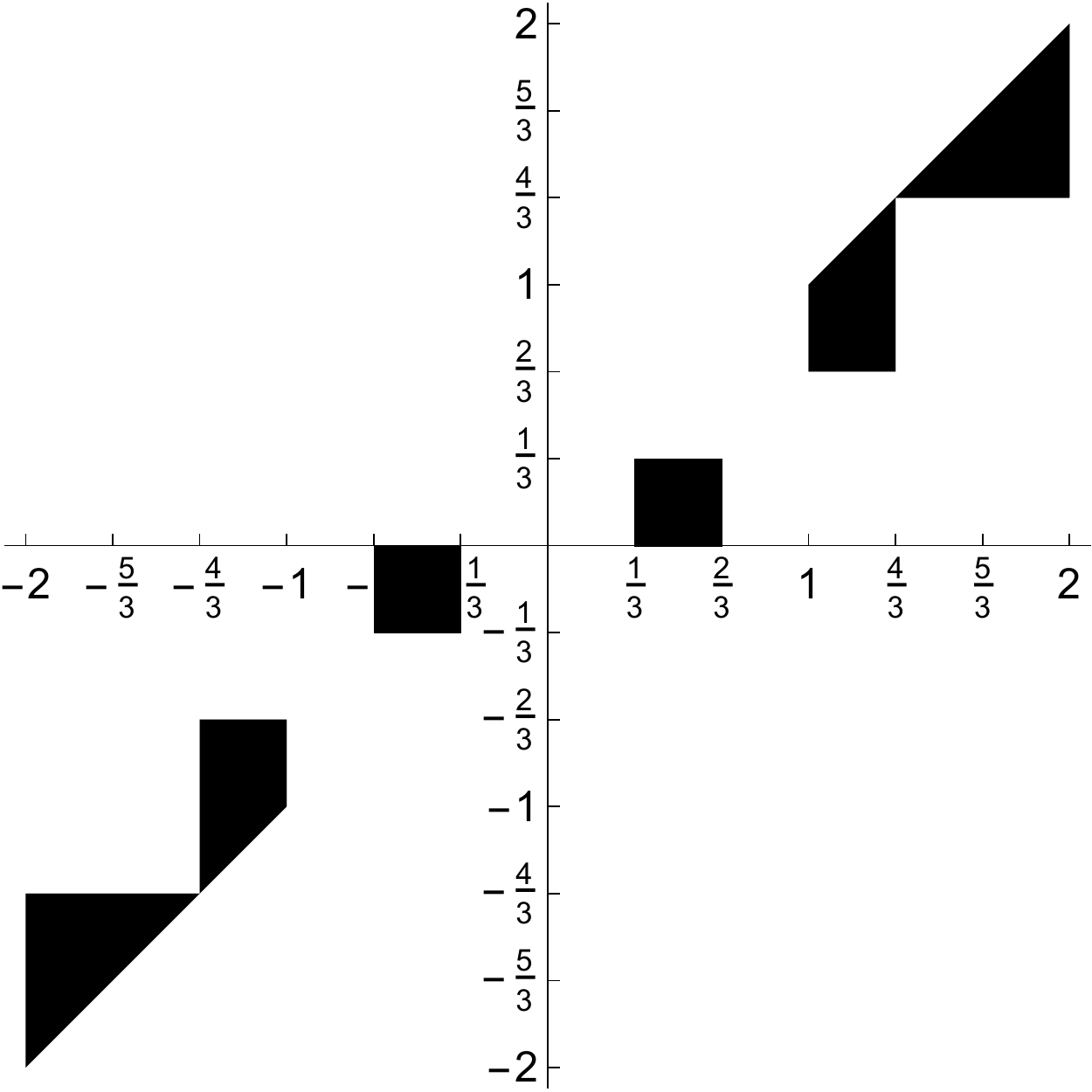}}
\end{picture}&
\setlength{\unitlength}{115bp}
\begin{picture}(1,1)(0,-.25)
\put(0,0){\includegraphics[width=\unitlength]{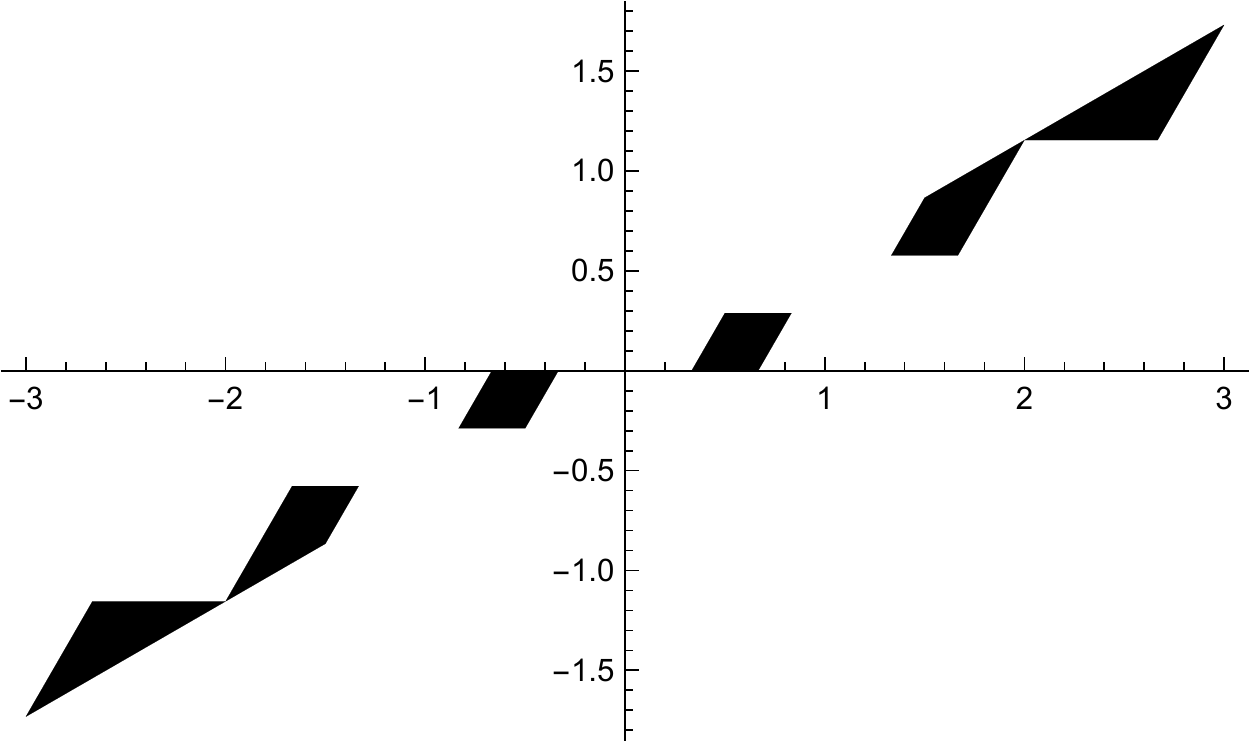}}
\end{picture}&
\setlength{\unitlength}{115bp}
\begin{picture}(1,1)(0,0)
\put(0,0){\includegraphics[width=\unitlength]{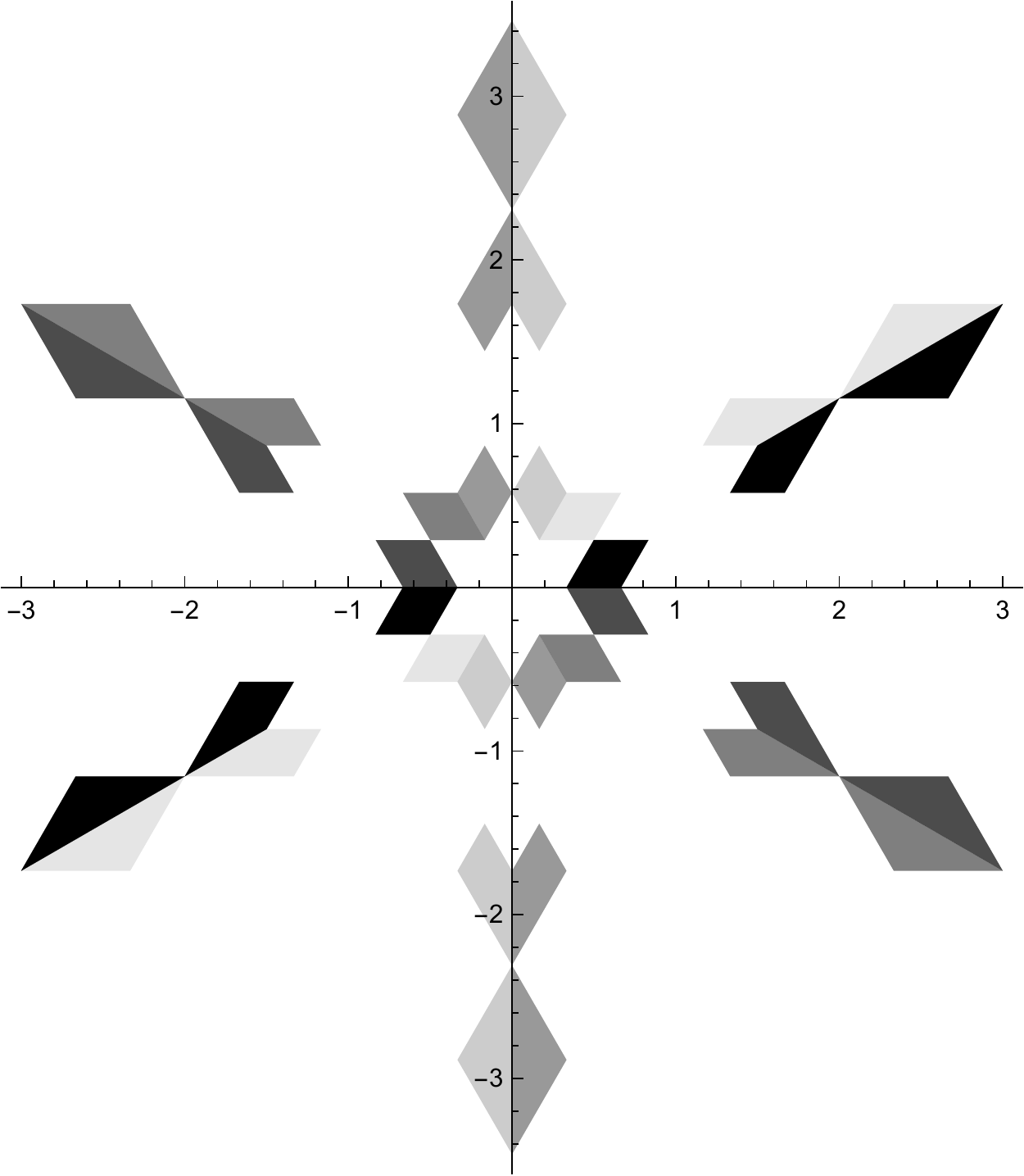}}
\end{picture}\\
D&W^{\rm p3m1}=LD&\cup_{S\in\mathcal B^{\rm p3m1}} SW^{\rm p3m1}
\end{array}$
\caption{Building a simple wavelet set for p3m1 and p31m}
\label{p31md2}
\end{figure}

\end{document}